\documentclass[12pt,a4paper]{article}

\usepackage{amsmath,amstext,amssymb,amscd}
\usepackage{graphicx}

\newcommand{\Xcomment}[1]{}

\newtheorem{theorem}{Theorem}[section]
\newtheorem{lemma}[theorem]{Lemma}

\newtheorem{prop}[theorem]{Proposition}

\makeatletter \@addtoreset{equation}{section} \makeatother

\newenvironment{proof}{\noindent{\bf Proof}\/}%
{\hfill$\qed$\medskip}

\def\qed{ \ \vrule width.1cm height.3cm depth0cm}

\newenvironment{numitem1}{\refstepcounter{equation}\begin{enumerate}%
\item[(\thesection.\arabic{equation})]}{\end{enumerate}}

\newcommand{\refeq}[1]{(\ref{eq:#1})}  % reference to equation

 \makeatletter
\renewcommand{\section}{\@startsection{section}{1}{0pt}%
{-3.5ex plus -1ex minus -.2ex}{2.3ex plus .2ex}%
{\normalfont\Large}}
 \makeatother

 \makeatletter
\renewcommand{\subsection}{\@startsection{subsection}{2}{0pt}%
{-3.0ex plus -1ex minus -.2ex}{-1.5ex plus .2ex}%
{\normalfont\normalsize\bf}}
 \makeatother

\newcommand{\SEC}[1]{\ref{sec:#1}}  % reference to section
\newcommand{\SSEC}[1]{\ref{ssec:#1}}  % reference to subsection

\def\Rset{{\mathbb R}}
\def\Kset{{\mathbb K}}

\def\Escr{{\cal E}}

\def\Lscr{{\cal L}}
\def\Mscr{{\cal M}}
\def\Nscr{{\cal N}}
\def\Oscr{{\cal O}}
\def\Pscr{{\cal P}}

\def\Sscr{{\cal S}}

\def\Zscr{{\cal Z}}

\def\tilde{\widetilde}

\def\bar{\overline}
\def\eps{\epsilon}

\def\Path{{\rm Path}}

\def\Inter{{\rm Int}}

\def\Iw{{I^\circ}}
\def\Jw{{J^\circ}}
\def\Ib{{I^\bullet}}
\def\Jb{{J^\bullet}}

\def\Yr{{Y^{\rm r}}}
\def\Yc{{Y^{\rm c}}}

\def\Pstan{\Pscr^{\rm st}}
\def\sgn{{\rm sgn}}

\def\precast{\prec^\ast}

\def\horvert{\;\hbox{\unitlength=1mm\begin{picture}(2,3)%
\put(0,3){\line(1,0){2}}\put(2,0){\line(0,1){3}}%
\end{picture}}\;}

\def\verthor{\;\hbox{\unitlength=1mm\begin{picture}(2,3)%
\put(0,0){\line(0,1){3}}\put(0,0){\line(1,0){2}}%
\end{picture}}\;}

  \Xcomment{

  }

\begin{document}

\baselineskip=15pt
\parskip=2pt

\title{Two statements on path systems \\related to quantum minors}

\author{
Vladimir~I.~Danilov\thanks{Central Institute of Economics and Mathematics of
the RAS, 47, Nakhimovskii Prospect, 117418 Moscow, Russia; emails:
danilov@cemi.rssi.ru}
  \and
Alexander~V.~Karzanov\thanks{Institute for System Analysis at the FRC Computer
Science and Control of the RAS, 9, Prospect 60 Let Oktyabrya, 117312 Moscow,
Russia; email: sasha@cs.isa.ru. Corresponding author.}
  }

\date{}

 \maketitle

 \begin{abstract}
In~\cite{DK} we gave a complete combinatorial characterization of homogeneous
quadratic identities for minors of quantum matrices. It was obtained as a
consequence of results on minors of matrices of a special sort, the so-called
\emph{path matrices} $\Path_G$ generated by paths in special planar directed
graphs $G$.

In this paper we prove two assertions that were stated but left unproved
in~\cite{DK}. The first one says that any minor of $\Path_G$ is determined by a
system of disjoint paths, called a \emph{flow}, in $G$ (generalizing a similar
result of Lindstr\"om's type for the path matrices of Cauchon graphs
in~\cite{cast1}). The second, more sophisticated, assertion concerns certain
transformations of pairs of flows in $G$.

 \medskip



 \end{abstract}

\parskip=3pt

%----------------------- Sec. 1

\section{\Large Introduction}  \label{sec:intr}

This paper is a supplement to~\cite{DK} where we developed a graph theoretic
construction (borrowing an idea of~\cite{cast1}) that was used as the main tool
to obtain a complete combinatorial characterization for the variety of
homogeneous quadratic identities on minors of quantum matrices.

(Recall that when speaking of the \emph{algebra of $m\times n$ quantum
matrices}, one means the quantized coordinate ring
$\Oscr_q(\Mscr_{m,n}(\Kset))$ of $m\times n$ matrices over a field $\Kset$,
where $q$ is a nonzero element of $\Kset$. In other words, one considers the
$\Kset$-algebra generated by indeterminates  $x_{ij}$\; ($i\in[m],\,j\in[n]$)
satisfying Manin's relations~\cite{man}: for $i<\ell\le m$ and $j<k\le n$,
   \begin{gather}
   x_{ij}x_{ik}=qx_{ik}x_{ij},\qquad x_{ij}x_{\ell j}=qx_{\ell j}x_{ij},
                                              \label{eq:xijkl}\\
   x_{ik}x_{\ell j}=x_{ \ell j}x_{ik}\quad \mbox{and}\quad
    x_{ij}x_{\ell k}-x_{\ell k}x_{ij}=(q-q^{-1})x_{ik}x_{\ell j}.  \nonumber
    \end{gather}
Hereinafter for a positive integer $n'$, \;$[n']$ denotes $\{1,2,\ldots,n'\}$.
Another useful algebraic construction is the $m\times n$ \emph{quantum affine
space}, which is the $\Kset$-algebra generated by indeterminates $t_{ij}$\;
($i\in[m],\,j\in[n]$) subject to ``simpler'' commutation relations:
  \begin{eqnarray}
  t_{ij}t_{i'j'}=&qt_{i'j'}t_{ij}&\quad \mbox{if either $i=i'$ and $j<j'$,
     or $i<i'$ and $j=j'$},  \label{eq:trelat} \\
     =&t_{i'j'}t_{ij}& \quad\mbox{otherwise}.) \nonumber
  \end{eqnarray}

In this paper we prove two auxiliary theorems that were essentially used, but
left unproved, in~\cite{DK} (namely, Theorems~3.1 and~4.4 there). They concern
the class of edge-weighted planar graphs introduced in~\cite{DK} (under the
name of ``grid-shaped graphs''); in this paper we call they \emph{SE-graphs}. A
special case of these graphs is formed by the \emph{Cauchon graphs} introduced
in~\cite{cast1} in connection with the \emph{Cauchon diagrams} of~\cite{cach}.
The first theorem, viewed as a quantum analog of Lindstr\"om Lemma, is a direct
extension to the SE-graphs $G$ of the corresponding result established for
Cauchon graphs in~\cite{cast1}. It considers a matrix in which each entry is
represented as the sum of weights of paths connecting a certain pair of
vertices of $G$, called the \emph{path matrix} of $G$ and denoted by $\Path_G$.
The theorem asserts that any (quantized) minor of $\Path_G$ can be expressed
via systems of disjoint paths of $G$ connecting corresponding sets of vertices.
We refer to a system of this sort as a \emph{flow} in $G$.

The proof of the main result in~\cite{DK} (which can be regarded as a quantum
analog of a characterization of quadratic identities for the commutative case
in~\cite{DKK}) is based on a method of handling certain pairs of flows, called
\emph{double flows}, in an SE-graph $G$. An important ingredient of that proof
is a transformation of a double flow $(\phi,\phi')$ into another double flow
$(\psi,\psi')$ by use of an \emph{ordinary exchange operation}. The second
theorem that we are going to prove in this paper says that under such a
transformation the weight of a current double flow is multiplied by $q$ or
$q^{-1}$.

The paper is organized as follows. Section~\SEC{prelim} contains basic
definitions and formulates the first theorem. Section~\SEC{double} describes
exchange operations on double flows and formulates the second theorem.
Section~\SEC{two_paths} elaborates technical tools needed to prove the
theorems. It considers certain paths $P,Q$ in $G$ and describes possible
relations between the weights of the ordered pairs $(P,Q)$ and $(Q,P)$; this is
close to a machinery in~\cite{cast1,cast2}. The announced first and second
theorems are proved in Sections~\SEC{q_determ} and~\SEC{exchange},
respectively.

%----------------------- Sec. 2

\section{\Large Preliminaries}  \label{sec:prelim}

We start with basic definitions and some elementary properties.
\medskip

\noindent\textbf{Paths in graphs.} Throughout, by a \emph{graph} we mean a
directed graph. A \emph{path} in a graph $G=(V,E)$ (with vertex set $V$ and
edge set $E$) is a sequence $P=(v_0,e_1,v_1,\ldots,e_k,v_k)$ such that each
$e_i$ is an edge connecting vertices $v_{i-1},v_i$. An edge $e_i$ is called
\emph{forward} if it is directed from $v_{i-1}$ to $v_i$, denoted as
$e_i=(v_{i-1},v_i)$, and \emph{ backward} otherwise (when $e_i=(v_i,v_{i-1})$).
The path $P$ is called {\em directed} if it has no backward edge, and {\em
simple} if all vertices $v_i$ are different. When $k>0$ and $v_0=v_k$, ~$P$ is
called a \emph{cycle}, and called a \emph{simple cycle} if, in addition,
$v_1,\ldots,v_k$ are different. When it is not confusing, we may use for $P$
the abbreviated notation via vertices: $P=v_0v_1\ldots v_k$, or via edges:
$P=e_1e_2\ldots e_k$.

Also, using standard terminology in graph theory, for a directed edge
$e=(u,v)$, we say that $e$ \emph{leaves} $u$ and \emph{enters} $v$, and that
$u$ is the \emph{tail} and $v$ is the \emph{head} of $e$.
\medskip

\noindent\textbf{SE-graphs.} A graph $G=(V,E)$ of this sort (also denoted as
$(V,E;R,C)$) is defined by the following conditions:

(i) $G$ is planar (with a fixed layout in the plane);

(ii) $G$ has edges of two types: \emph{horizontal} edges, or \emph{H-edges},
which are directed from left to right, and \emph{vertical} edges, or
\emph{V-edges}, which are directed downwards (so each edge points either
\emph{south} or \emph{east}, justifying the term ``SE-graph'');

(iii) $G$ has two distinguished subsets of vertices: set $R=\{r_1,\ldots,r_m\}$
of \emph{sources} and set $C=\{c_1,\ldots,c_n\}$ of \emph{sinks}; moreover,
$r_1,\ldots,r_m$ are disposed on a vertical line, in this order upwards, and
$c_1,\ldots,c_n$ are disposed on a horizontal line, in this order from left to
right;

(iv) each vertex (and each edge) of $G$ belongs to a directed path from $R$ to
$C$.

The set $V-(R\cup C)$ if \emph{inner} vertices of an SE-graph $G=(V,E)$ is
denoted by $W=W_G$. An example of SE-graphs with $m=3$ and $n=4$ is drawn in
the picture:

\vspace{0cm}
\begin{center}
\includegraphics{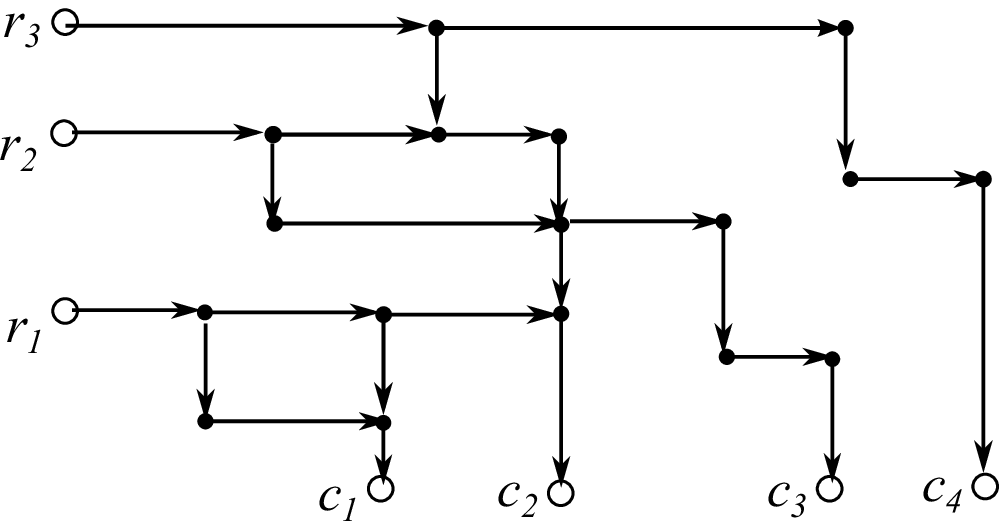}
\end{center}
\vspace{0cm}

Each inner vertex $v\in W$ is regarded as an indeterminate (generator), and we
assign a weight $w(e)$ to each edge $e$ in a way similar to the assignment for
Cauchon graphs in~\cite{cast1}. More precisely, for $e=(u,v)\in E$,
  \begin{numitem1} \label{eq:edge_weight}
  \begin{itemize}
  \item[(i)] $w(e):=v$ if $e$ is an H-edge with $u\in R$;
  \item[(ii)] $w(e):=u^{-1}v$ if $e$ is an H-edge with $u\in W$;
  \item[(iii)] $w(e):=1$ if $e$ is a V-edge.
    \end{itemize}
  \end{numitem1}
This gives rise to defining the weight $w(P)$ of a directed path
$P=e_1e_2\ldots e_k$ in $G$, to be the ordered (from left to right) product
  \begin{equation} \label{eq:wP}
  w(P)=w(e_1)w(e_2)\cdots w(e_k).
  \end{equation}

Then $w(P)$ is a Laurent monomial in elements of $W$. Note that when $P$ begins
in $R$ and ends in $C$, its weight can also be expressed in the following
useful form; cf.~\cite[Prop.~3.1.8]{cast2}. Let $u_1,v_1,u_2,v_2,\ldots,
u_{d-1},v_{d-1},u_d$ be the sequence of vertices where $P$ makes turns; namely,
$P$ changes the horizontal direction to the vertical one at each $u_i$, and
conversely at each $v_i$. Then (due to the ``telescopic effect'' caused
by~\refeq{edge_weight}(ii)),
  \begin{equation} \label{eq:telescop}
  w(P)=u_1v_1^{-1}u_2v_2^{-1}\cdots u_{d-1}v_{d-1}^{-1} u_d.
  \end{equation}

We assume that the generators $W$ obey (quasi)commutation laws somewhat similar
to those for the quantum affine space (cf.~\refeq{trelat}); namely,
  \begin{numitem1} \label{eq:commut_in_G}
  for distinct $u,v\in W$,
  \begin{itemize}
  \item[(i)] if there is a directed \emph{horizontal} path from $u$ to $v$ in $G$, then
$uv=qvu$;
  \item[(ii)] if there is a directed \emph{vertical} path from $u$ to $v$ in $G$, then
$vu=quv$;
  \item[(iii)] otherwise $uv=vu$.
  \end{itemize}
  \end{numitem1}

\noindent\textbf{Quantum minors.}  \label{ssec:quant_minor} It will be
convenient for us to visualize matrices in the Cartesian form: for an $m\times
n$ matrix $A=(a_{ij})$, the row indices $i=1,\ldots,m$ are assumed to increase
upwards, and the column indices $j=1,\ldots,n$ from left to right.

We denote by $A(I|J)$ the submatrix of $A$ whose rows are indexed by
$I\subseteq[m]$, and columns indexed by $J\subseteq[n]$. Let $|I|=|J|=:k$, and
let $I$ consist of $i_1<\ldots<i_k$ and $J$ consist of $j_1<\ldots<j_k$. Then
the $q$-\emph{determinant} of $A(I|J)$, or the $q$-\emph{minor} of $A$ for
$(I|J)$, is defined as
  \begin{equation} \label{eq:qminor}
  [I|J]_{A,q}:=\sum_{\sigma\in S_k} (-q)^{\ell(\sigma)}
    \prod_{d=1}^{k} a_{i_dj_{\sigma(d)}},
    \end{equation}
where, in the noncommutative case, the product under $\prod$ is ordered by
increasing $d$, and $\ell(\sigma)$ denotes the \emph{length} (number of
inversions) of a permutation $\sigma$. In the minor notation $[I|J]_{A,q}$, the
terms $A$ and/or $q$ may be omitted when they are clear from the context.
\medskip

\noindent\textbf{Path matrices.} An important construction in~\cite{cast1}
associates to a Cauchon graph $G$ a certain matrix, called the path matrix of
$G$, which has a nice property of Lindstr\"om's type: the $q$-minors of this
matrix correspond to appropriate systems of disjoint paths in $G$.

This is extended to an arbitrary SE-graph $G=(V,E;R,C)$. More precisely, let
$m:=|R|$ and $n:=|C|$. As before, $w=w_G$ denotes the edge weights in $G$
defined by~\refeq{edge_weight}. For $i\in[m]$ and $j\in[n]$, we denote the set
of directed paths from $r_i$ to $c_j$ in $G$ by $\Phi_G(i|j)$.
  \medskip

\noindent\textbf{Definition.} The \emph{path matrix} $\Path_G$ associated to
$G$ is the $m\times n$ matrix whose entries are defined by
  \begin{equation} \label{eq:Mat}
  \Path_G(i|j):=\sum\nolimits_{P\in\Phi_G(i|j)} w(P), \qquad (i,j)\in [m]\times [n],
  \end{equation}
In particular, $\Path_G(i|j)=0$ if $\Phi_G(i|j)=\emptyset$.
  \smallskip

Thus, the entries of $\Path_G$ belong to the $\Kset$-algebra $\Lscr_G$ of
Laurent polynomials generated by the set $W$ of inner vertices of $G$ subject
to relations~\refeq{commut_in_G}. (Note also that $\Path_G$ is a
$q$-\emph{matrix}, i.e., its entries obey Manin's relations;
see~\cite[Th.~3.2]{DK}).
\smallskip

\noindent\textbf{Definition.} Let $\Escr^{m,n}$ denote the set of pairs $(I|J)$
such that $I\subseteq [m]$, $J\subseteq [n]$ and $|I|=|J|$. Borrowing
terminology from~\cite{DKK}, we say that for $(I|J)\in\Escr^{m,n}$, a set
$\phi$ of pairwise \emph{disjoint} directed paths from the source set
$R_I:=\{r_i\;\colon i\in I\}$ to the sink set $C_J:=\{c_j\;\colon j\in J\}$ in
$G$ is an $(I|J)$-\emph{flow}. The set of $(I|J)$-flows is denoted by
$\Phi(I|J)=\Phi_G(I|J)$.
  \medskip

We throughout assume that the paths forming $\phi$ are ordered by increasing
the source indices. Namely, if $I$ consists of $i(1)<i(2)<\ldots< i(k)$ and $J$
consists of $j(1)<j(2)<\ldots<j(k)$, then $\ell$-th path $P_\ell$ in $\phi$
begins at $r_{i(\ell)}$, and therefore, $P_\ell$ ends at $c_{j(\ell)}$ (which
easily follows from the planarity of $G$, the ordering of sources and sinks in
the boundary of $G$ and the fact that the paths in $\phi$ are disjoint). We
write $\phi=(P_1,P_2,\ldots,P_k)$ and (similar to path systems in~\cite{cast1})
define the weight of $\phi$ to be the ordered product
  \begin{equation} \label{eq:w_phi}
  w(\phi):=w(P_1)w(P_2)\cdots w(P_k).
  \end{equation}

Our first theorem is a direct extension of a $q$-analog of Lindstr\"om's Lemma
shown for Cauchon graphs in~\cite[Th.~4.4]{cast1}; it gives a relationship
between flows and minors of path matrices.
  \begin{theorem} \label{tm:Lind}
Let $G$ be an SE-graph with $m$ sources and $n$ sinks. Then for the path matrix
$\Path=\Path_G$ and for any $(I|J)\in \Escr^{m,n}$, there holds
  \begin{equation} \label{eq:Lind}
[I|J]_{\Path,q}=\sum\nolimits_{\phi\in\Phi(I|J)} w(\phi).
  \end{equation}
  \end{theorem}
This theorem (stated in~\cite[Th.~3.1]{DK}) is proved in
Section~\SEC{q_determ}.

%----------------------- Sec. 3

\section{\Large Double flows, matchings, and exchange operations}  \label{sec:double}

A study of quadratic identities for minors of quantum matrices in~\cite{DK} is
reduced to handling ordered products of minors of the path matrices of
SE-graphs $G$, and further, in view of Theorem~\ref{tm:Lind}, to handling
ordered pairs of flows in $G$. On this way, a crucial role is played by
exchange operations on pairs of flows. To describe them, we first need some
definitions and conventions.

Let $G=(V,E;R,C)$ be an SE-graph  with $|R|=m$ and $|C|=n$. For
$(I|J),(I'|J')\in \Escr^{m,n}$, consider an $(I|J)$-flow $\phi$ and an
$(I'|J')$-flow $\phi'$ in $G$. We call the ordered pair $(\phi,\phi')$ a
\emph{double flow} in $G$. Define
   \begin{gather}
   \Iw:=I-I',\quad \Jw:=J-J',\quad \Ib:=I'-I,\quad \Jb:=J'-J,
                          \label{eq:white-black} \\
   \Yr:=\Iw\cup\Ib\quad \mbox{and}\quad \Yc:=\Jw\cup\Jb. \nonumber
   \end{gather}

Note that $|I|=|J|$ and $|I'|=|J'|$ imply that $|\Yr|+|\Yc|$ is even and that
  \begin{equation} \label{eq:balancIJ}
 |\Iw|-|\Ib|=|\Jw|-|\Jb|.
  \end{equation}
It is convenient for us to interpret $\Iw$ and $\Ib$ as the sets of
\emph{white} and \emph{black} elements of $\Yr$, respectively, and similarly
for $\Jw,\Jb,\Yc$, and to visualize these objects by use of a \emph{circular
diagram} $D$ in which the elements of $\Yr$ (resp. $\Yc$) are disposed in the
increasing order from left to right in the upper (resp. lower) half of a
circumference $O$. For example if, say, $\Iw=\{3\}$, $\Ib=\{1,4\}$,
$\Jw=\{2',5'\}$ and $\Jb=\{3',6',8'\}$, then the diagram is viewed as in the
left fragment of the picture below. (Sometimes, to avoid a possible mess
between elements of $\Yr$ and $\Yc$, and when it leads to no confusion, we
denote elements of $\Yc$ with primes.)

\vspace{0cm}
\begin{center}
\includegraphics{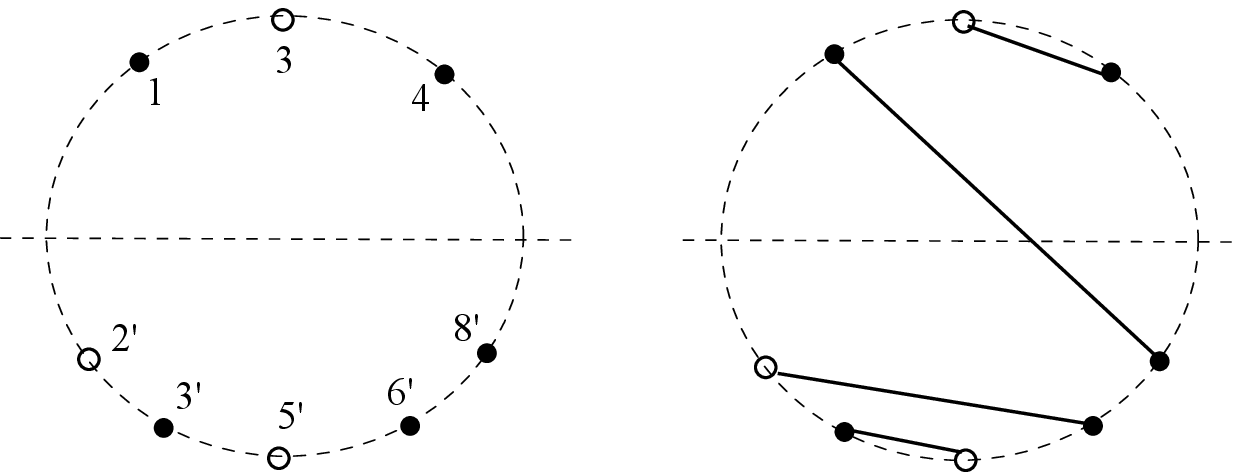}
\end{center}
\vspace{-0.3cm}

We refer to the quadruple $(I|J,I'|J')$ as a \emph{cortege}, and to
$(\Iw,\Ib,\Jw,\Jb)$ as the \emph{refinement} of $(I|J,I'|J')$, or as a
\emph{refined cortege}.
 \smallskip

Let $M$ be a partition of $\Yr\sqcup \Yc$ into 2-element sets (recall that
$A\sqcup B$ denotes the disjoint union of sets $A,B$). We refer to $M$ as a
\emph{perfect matching} on $\Yr\sqcup \Yc$, and to its elements as
\emph{couples}.

Also we say that $\pi\in M$ is: an $R$-\emph{couple} if $\pi\subseteq \Yr$, a
$C$-\emph{couple} if $\pi\subseteq \Yc$, and an $RC$-\emph{couple} if $|\pi\cap
\Yr|=|\pi\cap \Yc|=1$ (as though $\pi$ ``links'' two sources, two sinks, and
one source and one sink, respectively). \smallskip

\noindent\textbf{Definition.} A (perfect) matching $M$ as above is called a
\emph{feasible} matching for $(\Iw,\Ib,\Jw,\Jb)$ if:
  \begin{numitem1} \label{eq:feasM}
  \begin{itemize}
\item[(i)] for each $\pi=\{i,j\}\in M$, the elements $i,j$ have different
colors if $\pi$ is an $R$-couple or a $C$-couple, and have the same color if
$\pi$ is an $RC$-couple;
\item[(ii)] $M$ is \emph{planar}, in the sense that the chords connecting the
couples in the circumference $O$ are pairwise non-intersecting.
  \end{itemize}
  \end{numitem1}

The right fragment of the above picture illustrates an instance of feasible
matchings.

Return to a double flow $(\phi,\phi')$ as above. We associate to it a feasible
matching for $(\Iw,\Ib,\Jw,\Jb)$ as follows. Let $V_\phi$ and $E_\phi$,
respectively, denote the sets of vertices and edges of $G$ occurring in $\phi$,
and similarly for $\phi'$. Denote by $\langle U\rangle$ the subgraph of $G$
induced by the set of edges
   $$
   U:=E_\phi\triangle E_{\phi'},
   $$
writing $A\triangle B$ for the symmetric difference $(A-B)\cup(B-A)$ of sets
$A,B$. Then
  \begin{numitem1} \label{eq:degrees}
a vertex $v$ of $\langle U\rangle$ has degree 1 if $v\in R_{\Iw}\cup
R_{\Ib}\cup C_{\Jw}\cup C_{\Jb}$, and degree 2 or 4 otherwise.
  \end{numitem1}

We slightly modify $\langle U\rangle$ by splitting each vertex $v$ of degree 4
in $\langle U\rangle$ (if any) into two vertices $v',v''$ disposed in a small
neighborhood of $v$ so that the edges entering (resp. leaving) $v$ become
entering $v'$ (resp. leaving $v''$); see the picture.

\vspace{0cm}
\begin{center}
\includegraphics{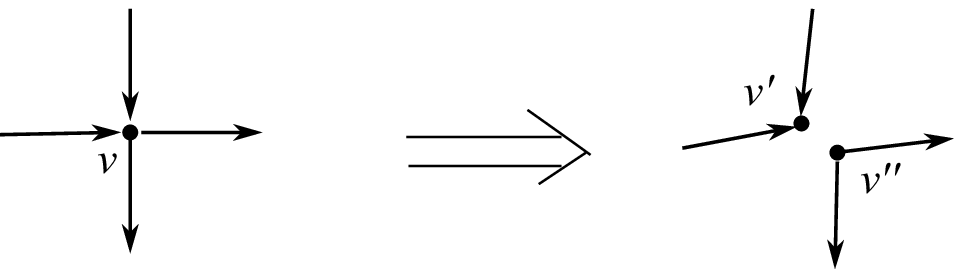}
\end{center}
\vspace{0cm}

The resulting graph, denoted as $\langle U\rangle '$, is planar and has
vertices of degree only 1 and 2. Therefore, $\langle U\rangle'$ consists of
pairwise disjoint (non-directed) simple paths $P'_1,\ldots,P'_k$ (considered up
to reversing) and, possibly, simple cycles $Q'_1,\ldots,Q'_d$. The
corresponding images of $P'_1,\ldots,P'_k$ (resp. $Q'_1,\ldots,Q'_d$) give
paths $P_1,\ldots,P_k$ (resp. cycles $Q_1,\ldots,Q_d$) in $\langle U\rangle$.
When $\langle U\rangle$ has vertices of degree 4, some of the latter paths and
cycles may be self-intersecting and may ``touch'', but not ``cross'', each
other. The following simple facts are shown in~\cite{DK}.
  \begin{lemma} \label{lm:P1Pk}
{\rm (i)} $k=(|\Iw|+|\Ib|+|\Jw|+|\Jb|)/2$;

{\rm(ii)} the set of endvertices of $P_1,\ldots,P_k$ is $R_{\Iw\cup\Ib}\cup
C_{\Jw\cup\Jb}$; moreover, each $P_i$ connects either $R_{\Iw}$ and $R_{\Ib}$,
or $C_{\Jw}$ and $C_{\Jb}$, or $R_{\Iw}$ and $C_{\Jw}$, or $R_{\Ib}$ and
$C_{\Jb}$;

{\rm(iii)} in each path $P_i$, the edges of $\phi$ and the edges of $\phi'$
have different directions (say, the former edges are all forward, and the
latter ones are all backward).
  \end{lemma}

Thus, each $P_i$ is represented as a concatenation $P_i^{(1)}\circ
P_i^{(2)}\circ\ldots\circ P_i^{(\ell)}$ of forwardly and backwardly directed
paths which are alternately contained in $\phi$ and $\phi'$. We call $P_i$ an
\emph{exchange path} (by a reason that will be clear later). The endvertices of
$P_i$ determine, in a natural way, a pair of elements of $\Yr\sqcup \Yc$,
denoted by $\pi_i$. Then $M:=\{\pi_1,\ldots,\pi_k\}$ is a perfect matching on
$\Yr\sqcup \Yc$.

Moreover, $M$ is a feasible matching for $(\Iw,\Ib,\Jw,\Jb)$, since
property~\refeq{feasM}(i) follows from Lemma~\ref{lm:P1Pk}(ii), and
property~\refeq{feasM}(ii) is provided by the fact that $P'_1,\ldots,P'_k$ are
pairwise disjoint simple paths in $\langle U\rangle'$. We denote $M$ as
$M(\phi,\phi')$, and for $\pi\in M$, denote by $P(\pi)$ the exchange path $P_i$
corresponding to $\pi$ (i.e., $\pi=\pi_i$).

Figure~\ref{fig:phi} illustrates an instance of $(\phi,\phi')$ for
$I=\{1,2,3\}$, $J=\{1',3',4'\}$, $I'=\{2,4\}$, $J'=\{2',3'\}$; here $\phi$ and
$\phi'$ are drawn by solid and dotted lines, respectively (in the left
fragment), the subgraph $\langle E_\phi\triangle E_{\phi'}\rangle$ consists of
three paths and one cycle (in the middle), and the circular diagram illustrates
$M(\phi,\phi')$ (in the right fragment).

\begin{figure}[htb]
\vspace{0.3cm}
\begin{center}
\includegraphics{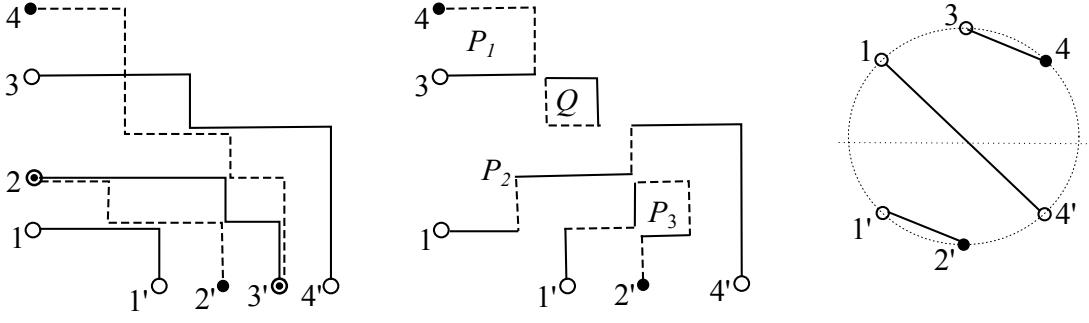}
\end{center}
\vspace{-0.3cm}
 \caption{ flows $\phi$ and $\phi'$ (left);  $\langle E_\phi\triangle
 E_{\phi'}\rangle$ (middle); $M(\phi,\phi')$ (right)}
 \label{fig:phi}
  \end{figure}

\noindent \textbf{Ordinary flow exchange operation.} Let us be given a double
flow $(\phi,\phi')$ for a cortege $(I|J,\,I'|J')$. Fix a couple $\pi=\{i,j\}\in
M(\phi,\phi')$. The operation w.r.t. $\pi$ rearranges $(\phi,\phi')$ into
another double flow $(\psi,\psi')$ for some $(\tilde I|\tilde J,\,\tilde
I'|\tilde J')$, as follows.

Consider the exchange path $P=P(\pi)$ corresponding to $\pi$, and let $\Escr$
be the set of edges of $P$. Define
  $$
  \tilde I:=I\triangle (\pi\cap \Yr), \quad \tilde I':=I'\triangle (\pi\cap \Yr), \quad
  \tilde J:=J\triangle (\pi\cap \Yc), \quad \tilde J':=J'\triangle (\pi\cap \Yc).
  $$

The following simple lemma is shown in~\cite{DK}.
  \begin{lemma} \label{lm:phi-psi}
The subgraph $\psi$ induced by $E_\phi\triangle\Escr$ gives a $(\tilde I|\tilde
J)$-flow, and the subgraph $\psi'$ induced by $E_{\phi'}\triangle \Escr$ gives
a $(\tilde I'|\tilde J')$-flow in $G$. Furthermore, $E_\psi\cup E_{\psi'}=
E_\phi\cup E_{\phi'}$,\; $E_\psi\triangle E_{\psi'}= E_\phi\triangle E_{\phi'}$
($=U$), and $M(\psi,\psi')=M(\phi,\phi')$.
  \end{lemma}

We call the transformation $(\phi,\phi')\stackrel{\pi}\longmapsto (\psi,\psi')$
in this lemma the \emph{ordinary flow exchange operation} for $(\phi,\phi')$
\emph{using} $\pi\in M(\phi,\phi')$ (or using $P(\pi)$). Clearly a similar
operation applied to $(\psi,\psi')$ using the same $\pi$ returns
$(\phi,\phi')$. The picture below illustrates flows $\psi,\psi'$ obtained from
$\phi,\phi'$ in Fig.~\ref{fig:phi} by the ordinary exchange operations using
the path $P_2$ (left) and the path $P_3$ (right).

\vspace{0.cm}
\begin{center}
\includegraphics{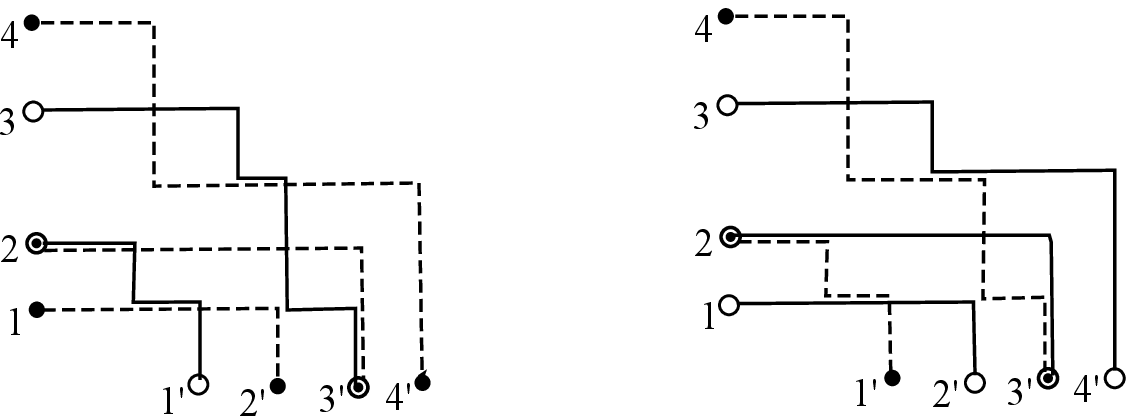}
\end{center}
\vspace{-0.3cm}

Now we formulate the second theorem of this paper; it will be proved in
Section~\SEC{exchange}.

  \begin{theorem} \label{tm:single_exch}
Let $\phi$ be an $(I|J)$-flow, and $\phi'$ an $(I'|J')$-flow in $G$. Let
$(\psi,\psi')$ be the double flow obtained from $(\phi,\phi')$ by the ordinary
flow exchange operation using a couple $\pi=\{f,g\}\in M(\phi,\phi')$. Then:

{\rm(i)} when $\pi$ is an $R$- or $C$-couple and $f<g$, we have
   \begin{eqnarray*}
   w(\phi)w(\phi')=qw(\psi)w(\psi') \quad &\mbox{in case}&\;\; f\in I\cup J; \\
 w(\phi)w(\phi')=q^{-1}w(\psi)w(\psi') \quad &\mbox{in case}&\;\;
                                                   f\in I'\cup J';
   \end{eqnarray*}

{\rm(ii)} when $\pi$ is an $RC$-couple, we have
$w(\phi)w(\phi')=w(\psi)w(\psi')$.
  \end{theorem}

%---------------------- Sec. 4

\section{\Large Commutation properties of paths}  \label{sec:two_paths}

This section contains auxiliary lemmas that will be used in the proofs of
Theorems~\ref{tm:Lind} and~\ref{tm:single_exch}. They deal with special pairs
$P,Q$ of paths in an SE-graph $G=(V,E;R,C)$ and compare the weights $w(P)w(Q)$
and $w(Q)w(P)$. Similar or close statements for Cauchon graphs are given
in~\cite{cast1,cast2}, and our method of proof is somewhat similar and rather
straightforward as well.

We need some terminology, notation and conventions.

When it is not confusing, vertices, edges, paths and other objects in $G$ are
identified with their corresponding images in the plane. We assume that the
sets $R$ and $C$ lie on the coordinate rays $(0,\Rset_{\ge 0})$ and
$(\Rset_{\ge 0},0)$, respectively (then $G$ is disposed within $\Rset^2_{\ge
0}$). The coordinates of a point $v$ in $\Rset^2$ (e.g., a vertex $v$ of $G$)
are denoted as $(\alpha(v),\beta(v))$. W.l.o.g. we may assume that two vertices
$u,v\in V$ have the same first (second) coordinate if and only if they belong
to a vertical (resp. horizontal) path in $G$, in which case $u,v$ are called
\emph{V-dependent} (resp. \emph{H-dependent}). When $u,v$ are V-dependent,
i.e., $\alpha(u)=\alpha(v)$, we say that $u$ is \emph{lower} than $v$ (and $v$
is \emph{higher} than $u$) if $\beta(u)< \beta(v)$. (In this case the
commutation relation $uv=qvu$ takes place.)

Let $P$ be a path in $G$. We denote: the first and last vertices of $P$ by
$s_P$ and $t_P$, respectively; the \emph{interior} of $P$ (the set of points of
$P-\{s_P,t_P\}$ in $\Rset^2$) by $\Inter(P)$; the set of horizontal edges of
$P$ by $E^H_P$; and the projection $\{\alpha(x)\;\colon x\in P\}$ by
$\alpha(P)$. Clearly if $P$ is directed, then $\alpha(P)$ is the interval
between $\alpha(s_P)$ and $\alpha(t_P)$.

For a directed path $P$, the following are equivalent: $P$ is non-vertical;
$E^H_P\ne \emptyset$; $\alpha(s_P)\ne\alpha(t_P)$; we will refer to such a $P$
as a \emph{standard} path.

For a standard path $P$, we will take advantage from a compact expression for
the weight $w(P)$. We call a vertex $v$ of $P$ \emph{essential} if either $P$
makes a turn at $v$ (changing the direction from horizontal to vertical or
back), or $v=s_P\not\in R$ and the first edge of $P$ is horizontal, or $v=t_P$
and the last edge of $P$ is horizontal. If $u_0,u_1,\ldots,u_k$ is the sequence
of essential vertices of $P$ in the natural order, then the weight of $P$ can
be expressed as
  \begin{equation} \label{eq:wP2}
  w(P)=u_0^{\sigma_0}u_1^{\sigma_1}\ldots u_k^{\sigma_k},
  \end{equation}
where $\sigma_i=1$ if $P$ makes a \horvert-turn at $u_i$ or if $i=k$, while
$\sigma_i=-1$ if $P$ makes a \verthor-turn at $u_i$ or if $i=0$. (Compare
with~\refeq{telescop} where a path from $R$ to $C$ is considered.) It is easy
to see that if $P$ does not begin in $R$, then its essential vertices are
partitioned into H-dependent pairs.

Throughout the rest of the paper, for brevity, we denote $q^{-1}$ by $\bar q$,
and for an inner vertex $v\in W$ regarded as a generator, we may denote
$v^{-1}$ by $\bar v$.
\smallskip

Now we start stating the desired lemmas on two directed paths $P,Q$. They deal
with the case when $P$ and $Q$ are \emph{weakly intersecting}, which means that
  \begin{equation} \label{eq:pathsPQ}
P\cap Q=\{s_P,t_P\}\cap \{s_Q,t_Q\};
  \end{equation}
in particular, $\Inter(P)\cap\Inter(Q)=\emptyset$. For such $P,Q$, we say that
$P$ is \emph{lower} than $Q$ if there are points $x\in P$ and $y\in Q$ such
that $\alpha(x)=\alpha(y)$ and $\beta(x)<\beta(y)$ (this property does not
depend on the choice of $x,y$). We define the value $\varphi=\varphi(P,Q)$ by
the relation
  $$
  w(P)w(Q)=\varphi w(Q)w(P).
  $$
Obviously, $\varphi(P,Q)=1$ when $P$ or $Q$ is a V-path. In the lemmas below we
default assume that both $P,Q$ are standard.

  \begin{lemma} \label{lm:varphi=1}
Let $\{\alpha(s_P),\alpha(t_P)\}\cap\{\alpha(s_Q),\alpha(t_Q)\}\cap
\Rset_{>0}=\emptyset$. Then $\varphi(P,Q)=1$.
  \end{lemma}
  \begin{proof}
~Consider an essential vertex $u$ of $P$ and an essential vertex $v$ of $Q$.
Then for any $\sigma,\sigma'\in\{1,-1\}$, we have $u^\sigma
v^{\sigma'}=v^{\sigma'} u^\sigma$ unless $u,v$ are dependent.

Suppose that $u,v$ are V-dependent. From the hypotheses of the lemma it follows
that at least one of the following is true:
$\alpha(s_P)<\alpha(u)<\alpha(t_P)$, or $\alpha(s_Q)<\alpha(v)<\alpha(t_Q)$.
For definiteness assume the former. Then there is another essential vertex $z$
of $P$ such that $\alpha(z)=\alpha(u)=\alpha(v)$. Moreover, $P$ makes a
\horvert-turn an one of $u,z$, and a \verthor-turn at the other. Since $P\cap
Q=\emptyset$ (in view of~\refeq{pathsPQ}), the vertices $u,z$ are either both
higher or both lower than $v$. Let for definiteness $u,z$ occur in this order
in $P$; then $w(P)$ contains the terms $u,\bar z$. Let $w(Q)$ contain the term
$v^\sigma$ and let $uv^\sigma=\rho v^\sigma u$, where $\sigma\in\{1,-1\}$ and
$\rho\in\{q,\bar q\}$. Then $\bar z v^\sigma= \bar \rho v^\sigma \bar z$,
implying $u\bar z v^\sigma =v^\sigma u\bar z$. Hence the contributions to
$w(P)w(Q)$ and $w(Q)w(P)$ from the pairs using terms $u,z,v$ (namely
$\{u,v^\sigma\}$ and $\{\bar z,v^\sigma\}$) are equal.

Next suppose that $u,v$ are H-dependent. One may assume that
$\alpha(u)<\alpha(v)$. Then $Q$ contains one more essential vertex $y\ne v$
with $\beta(y)=\beta(v)=\beta(u)$. Also $\alpha(u)<\alpha(v)$ and $P\cap
Q=\emptyset$ imply $\alpha(u)<\alpha(y)$. Let for definiteness
$\alpha(y)<\alpha(v)$. Then $w(P)$ contains the terms $\bar y,v$, and we can
conclude that the contributions to $w(P)w(Q)$ and $w(Q)w(P)$ from the pairs
using terms $u,y,v$ are equal (using the fact that
$\alpha(u)<\alpha(y),\alpha(v)$).

These reasonings imply $\varphi(P,Q)=1$.
  \end{proof}

  \begin{lemma} \label{lm:asP=asQ}
Let $\alpha(s_P)=\alpha(s_Q)>0$ and $\alpha(t_P)\ne\alpha(t_Q)$. Let $P$ be
lower than $Q$. Then $\varphi(P,Q)=q$.
  \end{lemma}
  \begin{proof}
~Let $u$ and $v$ be the first essential vertices in $P$ and $Q$, respectively.
Then $\alpha(u)=\alpha(s_P)=\alpha(s_Q)=\alpha(v)$ (in view of
$\alpha(s_P)=\alpha(s_Q)>0$). Since $P$ is lower than $Q$, we have
$\beta(u)\le\beta(v)$. Moreover, this inequality is strong (since
$\beta(u)=\beta(v)$ is impossible in view of~\refeq{pathsPQ} and the obvious
fact that $u,v$ are the tails of first H-edges in $P,Q$, respectively).

Now arguing as in the above proof, we can conclude that the discrepancy between
$w(P)w(Q)$ and $w(Q)w(P)$ can arise only due to swapping the vertices $u,v$.
Since $u$ gives the term $\bar u$ in $w(P)$, and $v$ the term $\bar v$ in
$w(Q)$, the contribution from these vertices to $w(P)w(Q)$ and $w(Q)w(P)$ are
expressed as $\bar u\bar v$ and $\bar v \bar u$, respectively. Since
$\beta(u)<\beta(v)$, we have $\bar u \bar v=q\bar v \bar u$, and the result
follows.
  \end{proof}

  \begin{lemma} \label{lm:atP=atQ}
Let $\alpha(t_P)=\alpha(t_Q)$ and let either $\alpha(s_P)\ne\alpha(s_Q)$ or
$\alpha(s_P)=\alpha(s_Q)=0$. Let $P$ be lower than $Q$. Then $\varphi(P,Q)=q$.
  \end{lemma}
 \begin{proof}
~We argue in spirit of the proof of Lemma~\ref{lm:asP=asQ}. Let $u$ and $v$ be
the last essential vertices in $P$ and $Q$, respectively. Then
$\alpha(u)=\alpha(t_P)=\alpha(t_Q)=\alpha(v)$. Also $\beta(u)<\beta(v)$ (since
$P$ is lower than $Q$, and in view of~\refeq{pathsPQ} and the fact that $u,v$
are the heads of H-edges in $P,Q$, respectively). The condition on
$\alpha(s_P)$ and $\alpha(s_Q)$ imply that the discrepancy between $w(P)w(Q)$
and $w(Q)w(P)$ can arise only due to swapping the vertices $u,v$ (using
reasonings as in the proof of Lemma~\ref{lm:varphi=1}). Observe that $w(P)$
contains the term $u$, and $w(Q)$ the term $v$. So the generators $u,v$
contribute $uv$ to $w(P)w(Q)$, and $vu$ to $w(Q)w(P)$. Now $\beta(u)<\beta(v)$
implies $uv=qvu$, and the result follows.
 \end{proof}

  \begin{lemma} \label{lm:1atP=asQ}
Let $\alpha(t_P)=\alpha(s_Q)$ and $\beta(t_P)\ge\beta(s_Q)$. Then
$\varphi(P,Q)=q$.
  \end{lemma}
 \begin{proof}
~Let $u$ be the last essential vertex in $P$ and let $v,z$ be the first and
second essential vertices of $Q$, respectively (note that $z$ exists because of
$0<\alpha(s_Q)<\alpha(t_Q)$). Then
$\alpha(u)=\alpha(t_P)=\alpha(s_Q)=\alpha(v)<\alpha(z)$. Also
$\beta(u)\ge\beta(t_P) \ge \beta(s_Q)\ge \beta(v)=\beta(z)$. Let $Q'$ and $Q''$
be the parts of $Q$ from $s_Q$ to $z$ and from $z$ to $t_Q$, respectively. Then
$\alpha(P)\cap\alpha(Q'')=\emptyset$, implying $\varphi_{P,Q''}=1$ (using
Lemma~\ref{lm:varphi=1} when $Q''$ is standard). Hence
$\varphi_{P,Q}=\varphi_{P,Q'}$.

To compute $\varphi_{P,Q'}$, consider three possible cases.

(a) Let $\beta(u)>\beta(v)$. Then $u,v$ form the unique pair of dependent
essential vertices for $P,Q'$. Note that $w(P)$ contains the term $u$, and
$w(Q')$ contains the term $\bar v$. Since $\beta(u)>\beta(v)$, we have $u\bar
v=q\bar vu$, implying $\varphi_{P,Q'}=q$.

(b) Let $u=v$ and let $u$ be the unique essential vertex of $P$ (in other
words, $P$ is an H-path with $s_P\in R$). Note that $u=v$ and
$\beta(t_P)\ge\beta(s_Q)$ imply $t_Q=u=v=s_P$. Also $\alpha(u)<\alpha(z)$ and
$\beta(u)=\beta(z)$; so $u,z$ are dependent essential vertices for $P,Q'$ and
$uz=qzu$. We have $w(P)=u$ and $w(Q')=\bar u z$ (in view of $u=v$). Then $u\bar
u z=\bar u uz=q\bar u z u$ gives $\varphi_{P,Q'}=q$.

(c) Now let $u=v$ and let $y$ be the essential vertex of $P$ preceding $u$.
Then $t_Q=u=v=s_P$, ~$\beta(y)=\beta(u)=\beta(z)$, and
$\alpha(y)<\alpha(u)<\alpha(z)$. Hence $y,u,z$ are dependent, $w(P)$ contains
$\bar yu$, and $w(Q')=\bar uz$. We have
  $$
  \bar y u\bar u z= \bar y\bar u uz=(q\bar u\bar y)(qzu)
          =q^2 \bar u(\bar q z\bar y) u=q\bar u z\bar y u,
  $$
again obtaining $\varphi_{P,Q'}=q$.
 \end{proof}

   \begin{lemma} \label{lm:2atP=asQ}
Let $\alpha(t_P)=\alpha(s_Q)$ and  $\beta(t_P)<\beta(s_Q)$. Then
$\varphi(P,Q)=\bar q$.
  \end{lemma}
 \begin{proof}
~Let $u$ be the last essential vertex of $P$, and $v$ the first essential
vertex of $Q$. Then $\alpha(u)=\alpha(t_P)=\alpha(s_Q)=\alpha(v)$, and
$\beta(t_P)<\beta(s_Q)$ together with~\refeq{pathsPQ} implies
$\beta(u)<\beta(v)$. Also $w(P)$ contains $u$ and $w(Q)$ contains $\bar v$. Now
$u\bar v=\bar q \bar v u$ implies $\varphi_{P,Q}=\bar q$.
 \end{proof}

 \Xcomment{
  \begin{lemma} \label{lm:astP=astQ}
Let $\alpha(s_P)=\alpha(s_Q)> 0$ and $\alpha(t_P)=\alpha(t_Q)$. Let $P$ be
lower than $Q$. Then $\varphi(P,Q)=q^2$.
  \end{lemma}
 \begin{proof}
~In fact, this is a combination of Lemmas~\ref{lm:asP=asQ}
and~\ref{lm:atP=atQ}. More precisely, let $u,v$ (resp. $u',v'$) be the first
(resp. last) essential vertices of $P$ and $Q$, respectively. Then
$\alpha(u)=\alpha(v)$, ~$\beta(u)<\beta(v)$, ~$\alpha(u')=\alpha(v')$, and
$\beta(u')<\beta(v')$. Also $w(P)$ contains $\bar u,u'$ and $w(Q)$ contains
$\bar v,v'$. Now $\bar u\bar v=q\bar v\bar u$ and $u'v'=qv'u'$ imply
$\varphi(P,Q)=q^2$.
 \end{proof}
  }

%----------------------- Sec. 5

\section{\Large Proof of Theorem~\ref{tm:Lind}}  \label{sec:q_determ}

It can be conducted as a direct extension of the proof of a similar
Lindstr\"om's type result given by Casteels~\cite[Sec.~4]{cast1} for Cauchon
graphs. To make our description more self-contained, we outline the main
ingredients of the proof, leaving the details where needed to the reader.

Let $(I|J)\in\Escr^{m,n}$, $I=\{i(1)<\cdots <i(k)\}$ and
$J=\{j(1)<\cdots<j(k)\}$. Recall that an $(I|J)$-flow in an SE-graph $G$ (with
$m$ sources and $n$ sinks) consists of pairwise disjoint paths $P_1,\ldots,
P_k$ from the source set $R_I=\{r_{i(1)},\ldots,r_{i(k)}\}$ to the sink set
$C_J=\{c_{j(1)},\ldots,c_{j(k)}\}$, and (due to the planarity of $G$) we may
assume that each $P_d$ begins at $r_{i(d)}$ and ends at $c_{j(d)}$. Besides, we
are forced to deal with an arbitrary \emph{path system} $\Pscr=(P_1,\ldots,
P_k)$ in which for $i=1,\ldots,k$, ~$P_d$ is a directed path in $G$ beginning
at $r_{i(d)}$ and ending at $c_{j(\sigma(d))}$, where $\sigma(1),
\ldots,\sigma(k)$ are different, i.e., $\sigma=\sigma_\Pscr$ is a permutation
on $[k]$. (In particular, $\sigma_\Pscr$ is identical if $\Pscr$ is a flow.)

We naturally partition the set of all path systems for $G$ and $(I|J)$ into the
set $\Phi=\Phi_G(I|J)$ of $(I|J)$-flows and the rest $\Psi=\Psi_G(I|J)$
(consisting of those path systems that contain intersecting paths). The
following property easily follows from the planarity of $G$
(cf.~\cite[Lemma~4.2]{cast1}):
   \begin{numitem1} \label{eq:PiPi+1}
For any $\Pscr=(P_1,\ldots,P_k)\in\Psi$, there exist two \emph{consecutive}
intersecting paths $P_d,P_{d+1}$.
  \end{numitem1}

The $q$-\emph{sign} of a permutation $\sigma$ is defined by
   $$
   \sgn_q(\sigma):=(-q)^{\ell(\sigma)},
   $$
where $\ell(\sigma)$ is the length of $\sigma$ (see Sect.~\SEC{prelim}).

Now we start computing the $q$-minor $[I|J]$ of the matrix $\Path_G$ with the
following chain of equalities:
   \begin{eqnarray*}
   [I|J]&=& \sum\nolimits_{\sigma\in S_k} \sgn_q(\sigma)
        \left( \prod\nolimits_{d=1}^{k} \Path_G(i(d)|j(\sigma(d))\right) \\
   &=& \sum\nolimits_{\sigma\in S_k} \sgn_q(\sigma)
       \left( \prod\nolimits_{d=1}^{k} \left(
                     \sum(w(P)\;\colon P\in \Pscr_G(i(d)|j(\sigma(d))\right)\right) \\
   &=&\sum(\sgn_q(\sigma_\Pscr)w(\Pscr)\;\colon \Pscr\in\Phi\cup\Psi) \\
   &=&\sum(w(\Pscr)\;\colon \Pscr\in\Phi)
                        +\sum(\sgn_q(\sigma_\Pscr)w(\Pscr)\;\colon
                        \Pscr\in\Psi).
   \end{eqnarray*}

Thus, we have to show that the second sum in the last row is zero. It will
follow from the existence of an involution $\eta:\Psi\to\Psi$ without fixed
points such that for each $\Pscr\in\Psi$,
  \begin{equation} \label{eq:invol}
  \sgn_q(\sigma_\Pscr)w(\Pscr)=-\sgn_q(\sigma_{\eta(\Pscr)}) w(\eta(\Pscr)).
  \end{equation}

To construct the desired $\eta$, consider $\Pscr=(P_1,\ldots,P_k)\in\Psi$, take
the minimal $i$ such that $P_i$ and $P_{i+1}$ meet, take the last common vertex
$v$ of these paths, represent $P_i$ as the concatenation $K\circ L$, and
$P_{i+1}$ as $K'\circ L'$, so that $t_K=t_{K'}=s_L=s_{L'}=v$, and exchange the
portions $L,L'$ of these paths, forming $Q_i:=K\circ L'$ and $Q_{i+1}:=K'\circ
L$. Then we assign $\eta(\Pscr)$ to be obtained from $\Pscr$ by replacing
$P_i,P_{i+1}$ by $Q_i,Q_{i+1}$. It is routine to check that $\eta$ is indeed an
involution (with $\eta(\Pscr)\ne\Pscr$) and that
   \begin{equation} \label{eq:ell+1}
  \ell(\sigma_{\eta(\Pscr)})=\ell(\sigma_\Pscr)+1,
  \end{equation}
assuming w.l.o.g. that $\sigma(i)<\sigma(i+1)$. On the other hand, applying to
the paths $K,L,K',L'$ corresponding lemmas from Sect.~\SEC{two_paths} (among
Lemmas~\ref{lm:asP=asQ}--\ref{lm:1atP=asQ}), one can obtain
   \begin{multline*} %\label{eq:KLK'L'}
   \quad w(P_i)w(P_{i+1})=w(K)w(L)w(K')w(L')=qw(K)w(L)w(L')w(K')  \\
  =q^2 w(K)w(L')w(L)w(K')=qw(K)w(L')w(K')w(L)=qw(Q_i)w(Q_{i+1}),\quad
   \end{multline*}
whence $w(\Pscr)=qw(\eta(\Pscr))$. This together with~\refeq{ell+1} gives
   %
%  \begin{multline*}
   \begin{equation*}
  \sgn_q(\sigma_\Pscr)w(\Pscr)+\sgn_q(\sigma_{\eta(\Pscr)}) w(\eta(\Pscr))
  =(-q)^{\ell(\sigma_\Pscr)}q w(\eta(\Pscr))+(-q)^{\ell(\sigma_\Pscr)+1}
  w(\eta(\Pscr))   =0,
  \end{equation*}
yielding~\refeq{invol}, and the result follows. \hfill\qed

%----------------------- Sec. 6

\section{\Large Proof of Theorem~\ref{tm:single_exch}}  \label{sec:exchange}

Using notation as in the hypotheses of this theorem, we first consider the case
when

\begin{itemize}
\item[(C):] $\pi=\{f,g\}$ is a $C$-couple in $M(\phi,\phi')$ with $f<g$ and $f\in
J$.
  \end{itemize}
(Then $f\in\Jw$ and $g\in\Jb$.) We have to prove that
  \begin{equation} \label{eq:caseC}
  w(\phi)w(\phi')=qw(\psi)w(\psi')
  \end{equation}
The proof is given throughout Sects.~\SSEC{seglink}--\SSEC{degenerate}. The
other possible cases in Theorem~\ref{tm:single_exch} will be discussed in
Sect.~\SSEC{othercases}.

%---------------------------------
 \subsection{Snakes and links.} \label{ssec:seglink}
Let $Z$ be the exchange path determined by $\pi$ (i.e., $Z=P(\pi)$ in notation
of Sect.~\SEC{double}). It connects the sinks $c_f$ and $c_g$, which may be
regarded as the first and last vertices of $Z$, respectively. Then $Z$ is
representable as a concatenation $Z=\bar Z_1\circ Z_2\circ\bar Z_3\circ \ldots
\circ \bar Z_{k-1}\circ Z_k$, where $k$ is even, each $Z_i$ with $i$ odd (even)
is a directed path concerning $\phi$ (resp. $\phi'$), and $\bar Z_i$ stands for
the path reversed to $Z_i$. More precisely, let $z_0:=c_f$, ~$z_k:=c_g$, and
for $i=1,\ldots,k-1$, denote by $z_i$ the common endvertex of $Z_i$ and
$Z_{i+1}$. Then each $Z_i$ with $i$ odd is a directed path from $z_i$ to
$z_{i-1}$ in $\langle E_\phi- E_{\phi'}\rangle$, while each $Z_i$ with $i$ even
is a directed path from $z_{i-1}$ to $z_i$ in $\langle E_{\phi'}-
E_{\phi}\rangle$.

We refer to $Z_i$ with $i$ odd (even) as a \emph{white} (resp. \emph{black})
\emph{snake}.

Also we refer to the vertices $z_1,\ldots,z_{k-1}$ as the \emph{bends} of $Z$.
A bend $z_i$ is called a \emph{peak} (a \emph{pit}) if both path $Z_i,Z_{i+1}$
leave (resp. enter) $z_i$; then $z_1,z_3,\ldots,z_{k-1}$ are the peaks, and
$z_2,z_4,\ldots,z_{k-2}$ are the pits. Note that some peak $z_i$ and pit $z_j$
may coincide; in this case we say that $z_i,z_j$ are \emph{twins}.

The rests of flows $\phi$ and $\phi'$ consist of directed paths that we call
\emph{white} and \emph{black links}, respectively. More precisely, the white
(black) links correspond to the connected components of the subgraph $\phi$
(resp. $\phi'$) from which the interiors of all snakes are removed. So a link
connects either (a) a source and a sink (being a component of $\phi$ or
$\phi'$), or (b) a source and a pit, or (c) a peak and a sink, or (d) a pit and
a peak. We say that a link is \emph{unbounded} in case (a), \emph{semi-bounded}
in cases (b),(c), and \emph{bounded} in case (d). Note that
  \begin{numitem1} \label{eq:4paths}
a bend $z_i$ occurs as an endvertex in exactly four paths among snakes and
links, namely: either in two snakes and two links (of different colors), or in
four snakes $Z_i,Z_{i+1},Z_j,Z_{j+1}$ (when $z_i,z_j$ are twins).
  \end{numitem1}

We denote the sets of snakes and links (for $\phi,\phi',\pi$) by $\Sscr$ and
$\Lscr$, respectively; the corresponding subsets of white and black elements of
these sets are denoted as $\Sscr^\circ,\; \Sscr^\bullet,\; \Lscr^\circ,\;
\Lscr^\bullet$.

The picture below illustrates an example. Here $k=10$, the bends
$z_1,\ldots,z_9$ are marked by squares, the white and black snakes are drawn by
thin and thick solid zigzag lines, respectively, the white links
($L_1,\ldots,L_7$) by short-dotted lines, and the black links
($M_1,\ldots,M_6$) by long-dotted lines.

\vspace{0cm}
\begin{center}
\includegraphics{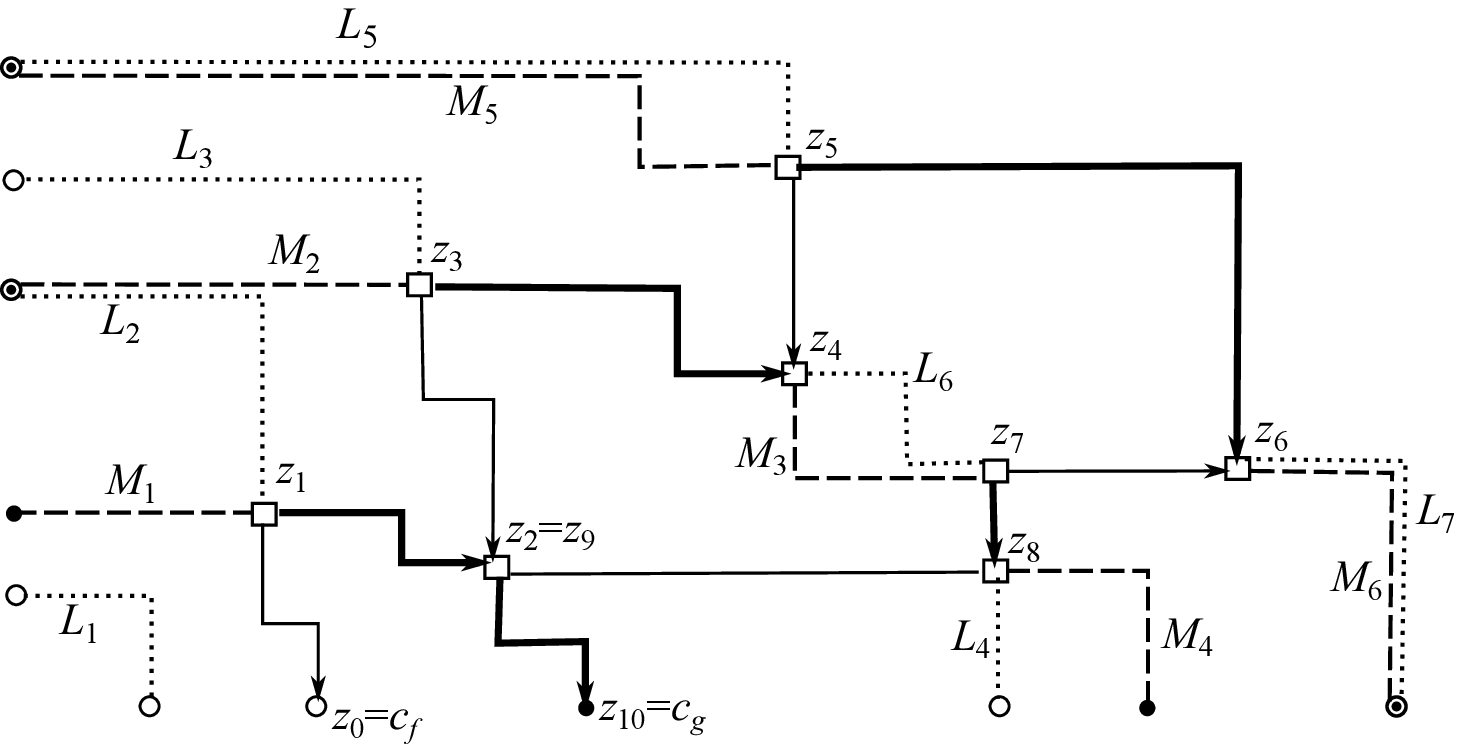}
\end{center}
\vspace{0cm}

The weight $w(\phi)w(\phi')$ of the double flow $(\phi,\phi')$ can be written
as the corresponding ordered product of the weights of snakes and links; let
$\Nscr$ be the string (sequence) of snakes and links in this product. The
weight of the double flow $(\psi,\psi')$ uses a string consisting of the same
snakes and links but taken in another order; we denote this string by
$\Nscr^\ast$.

We say that two elements among snakes and links are \emph{invariant} if they
occur in the same order in $\Nscr$ and $\Nscr^\ast$, and \emph{permuting}
otherwise. In particular, two links of different colors are invariant, whereas
two snakes of different colors are always permitting.

For example, observe that the string $\Nscr$ for the above illustration is
viewed as
  $$
L_1L_2Z_1L_3Z_3Z_9L_4L_5Z_5L_6Z_7L_7M_1Z_2Z_{10}M_2Z_4M_3Z_8M_4M_5Z_6M_6,
  $$
whereas $\Nscr^\ast$ is viewed as
  $$
L_1L_2Z_2Z_{10}L_3Z_4L_6Z_8L_4L_5Z_6L_7M_1Z_1M_2Z_3Z_9M_4M_5Z_5M_3Z_7M_6.
  $$

For $A,B\in\Sscr\cup\Lscr$, we write $A\prec B$ (resp. $A\prec^\ast B$) if $A$
occurs in $\Nscr$ (resp. in $\Nscr^\ast$) earlier than $B$. We define
$\varphi_{A,B}=\varphi_{B,A}:=1$ if $A,B$ are invariant, and define
$\varphi_{A,B}=\varphi_{B,A}$ by the relation
   \begin{equation} \label{eq:phiAB}
   w(A)w(B)=\varphi_{A,B} w(B)w(A).
   \end{equation}
if $A,B$ are permuting and $A\prec B$. Note that $\varphi_{A,B}$ is defined
somewhat differently than $\varphi(P,Q)$ in Sect.~\SEC{two_paths}.

For $A,B\in\Sscr\cup\Lscr$, we may use notation $(A,B)$ when $A,B$ are
permuting and $A\prec B$ (and may write $\{A,B\}$ when their orders by $\prec$
and $\prec^\ast$ are not important for us).

Our goal is to prove that in case~(C),
  \begin{equation}\label{eq:Pi=q}
  \prod(\varphi_{A,B}\;\colon A,B\in \Sscr\cup\Lscr)=q,
  \end{equation}
whence~\refeq{caseC} will immediately follow.

We first consider the \emph{non-degenerate} case. This means the following
restriction:
  \begin{numitem1} \label{eq:nondegenerate}
all coordinates $\alpha(z_1),\ldots,\alpha(z_{k-1}),
\alpha(c_1),\ldots,\alpha(c_n)$ of bends and sinks are different.
  \end{numitem1}

The proof of~\refeq{Pi=q} subject to~\refeq{nondegenerate} will consist of
three stages I, II, III where we compute the total contribution from the pairs
of links, the pairs of snakes, and the pairs consisting of one snake and one
link, respectively. As a consequence, the following three results will be
obtained (implying~\refeq{Pi=q}).
  \begin{prop} \label{pr:link-link}
In case~\refeq{nondegenerate}, the product $\varphi^I$ of the values
$\varphi_{A,B}$ over links $A,B\in\Lscr$ is equal to 1.
  \end{prop}
  \begin{prop} \label{pr:seg-seg}
In case~\refeq{nondegenerate}, the product $\varphi^{II}$ of the values
$\varphi_{A,B}$ over snakes $A,B\in\Sscr$ is equal to q.
  \end{prop}
  \begin{prop} \label{pr:seg-link}
In case~\refeq{nondegenerate}, the product $\varphi^{III}$ of the values
$\varphi_{A,B}$ where one of $A,B$ is a snake and the other is a link is equal
to 1.
  \end{prop}

These propositions are proved in Sects.~\SSEC{prop1}--\SSEC{prop3}. Sometimes
it will be convenient for us to refer to a white (black) snake/link concerning
$\phi,\phi',\pi$ as a $\phi$-snake/link (resp. a $\phi'$-snake/link), and
similarly for $\psi,\psi',\pi$.

%----------------
 \subsection{Proof of Proposition~\ref{pr:link-link}.} \label{ssec:prop1}
Under the exchange operation using $Z$, any $\phi$-link becomes a $\psi$-link
and any $\phi'$-link becomes a $\psi'$-link. The white links occur in $\Nscr$
earlier than the black links, and similarly for  $\Nscr^\ast$. Therefore, if
$A,B$ are permuting links, then they are of the same color. This implies that
$A\cap B=\emptyset$. Also each endvertex of any link either is a bend or
belongs to $R\cup C$. Then~\refeq{nondegenerate} implies that the sets
$\{\alpha(s_A),\alpha(t_A)\}\cap \Rset_{>0}$ and
$\{\alpha(s_B),\alpha(t_B)\}\cap \Rset_{>0}$ are disjoint. Now
Lemma~\ref{lm:varphi=1} gives $\varphi_{A,B}=1$, and the proposition follows.
\hfill\qed

%----------------
 \subsection{Proof of Proposition~\ref{pr:seg-seg}.} \label{ssec:prop2}

Consider two snakes $A=Z_i$ and $B=Z_j$, and let $A\prec B$. If $|i-j|>1$ then
$A\cap B=\emptyset$ and, moreover, $\{\alpha(s_A),\alpha(t_A)\}\cap
\{\alpha(s_B),\alpha(t_B)\}=\emptyset$ (since $Z$ is simple and in view
of~\refeq{nondegenerate}). This gives $\varphi_{A,B}=1$, by
Lemma~\ref{lm:varphi=1}.

Now let $|i-j|=1$. Then $A,B$ have different colors; hence $A$ is white and $B$
is black (in view of $A\prec B$). So $i$ is odd, and two cases are possible:
\smallskip

\noindent\underline{\emph{Case 1}:} ~$j=i+1$ and $z_i$ is a peak:
$z_i=s_A=s_B$;
  \smallskip

\noindent\underline{\emph{Case 2}:} ~$j=i-1$ and $z_{i-1}$ is a pit:
$z_{i-1}=t_A=t_B$.
  \smallskip

Cases 1,2 are divided into two subcases each.
 \smallskip

\noindent\underline{\emph{Subcase 1a}:} ~$j=i+1$ and $A$ is lower than $B$.
 \smallskip

\noindent\underline{\emph{Subcase 1b}:} ~$j=i+1$ and $B$ is lower than $A$.
 \smallskip

\noindent\underline{\emph{Subcase 2a}:} ~$j=i-1$ and $A$ is lower than $B$.
 \smallskip

\noindent\underline{\emph{Subcase 2b}:} ~$j=i-1$ and $B$ is lower than $A$.
 \smallskip

(Recall that for directed paths $P,Q$ satisfying~\refeq{pathsPQ}, $P$ is said
to be \emph{lower} than $Q$ if there are $x\in P$ and $y\in Q$ with
$\alpha(x)=\alpha(y)$ and $\beta(x)<\beta(y)$.) Subcases~1a--2b are illustrated
in the picture:

\vspace{-0cm}
\begin{center}
\includegraphics{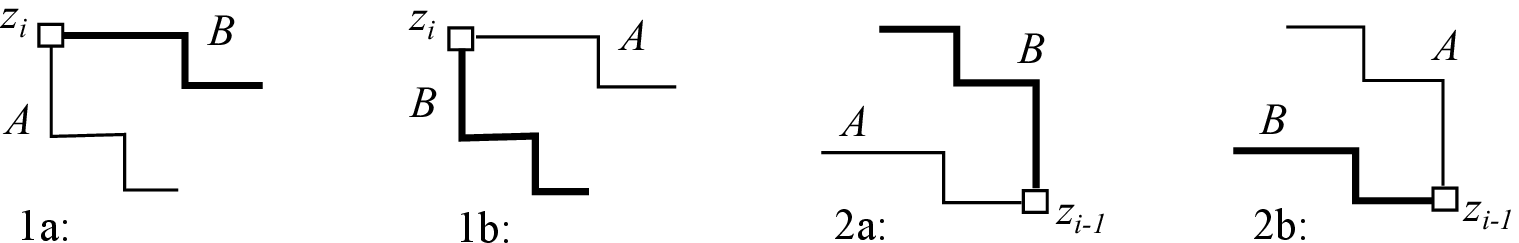}
\end{center}
\vspace{0cm}

Under the exchange operation using $Z$, any snake changes its color; so $A,B$
are permuting. Applying to $A,B$ Lemmas~\ref{lm:asP=asQ} and~\ref{lm:atP=atQ},
we obtain $\varphi_{A,B}=q$ in Subcases~1a,2a, and $\varphi_{A,B}=\bar q$ in
Subcases~1b,2b.

It is convenient to associate with a bend $z$ the number $\gamma(z)$ which is
equal to $+1$ if, for the corresponding pair $A\in\Sscr^\circ$ and
$B\in\Sscr^\bullet$ sharing $z$, ~$A$ is lower than $B$ (as in Subcases~1a,2a),
and equal to $-1$ otherwise (as in Subcases~1b,2b). Define
  \begin{equation} \label{eq:gammaZ}
  \gamma_Z:=\sum(\gamma(z)\;\colon z\;\; \mbox{a bend of}\;\; Z).
  \end{equation}
Then $\varphi^{II}=q^{\gamma_Z}$. Thus, $\varphi^{II}=q$ is equivalent to
   \begin{equation} \label{eq:gamma=1}
   \gamma_Z=1.
   \end{equation}

To show~\refeq{gamma=1}, we are forced to deal with a more general setting.
More precisely, let us turn $Z$ into simple cycle $D$ by combining the directed
path $Z_1$ (from $z_1$ to $z_0=c_f$) with the horizontal path from $c_f$ to
$c_g$ (to create the latter, we formally add to $G$ the horizontal edges
$(c_j,c_{j+1})$ for $j=f,\ldots,g-1$). The resulting directed path $\tilde Z$
from $z_1$ to $c_g=z_k$ is regarded as the new white snake replacing $Z_1$.
Then $\tilde Z_1$ shares the end $z_k$ with the black path $Z_k$; so $z_k$ is a
pit of $D$, and $\tilde Z$ is lower than $Z_k$. Thus, compared with $Z$, the
cycle $D$ acquires an additional bend, namely, $z_k$. We have $\gamma(z_k)=1$,
implying $\gamma_D=\gamma_Z+1$. Then~\refeq{gamma=1} is equivalent to
$\gamma_D=2$.

On this way, we come to a new (more general) setting by considering an
arbitrary simple (non-directed) cycle $D$ rather than a special path $Z$.
Moreover, instead of an SE-graph as before, we can work with a more general
directed planar graph $G$ in which any edge $e=(u,v)$ points arbitrarily within
the south-east sector, i.e., satisfies $\alpha(u)\le \alpha(v)$ and
$\beta(u)\ge \beta(v)$. We call $G$ of this sort a \emph{weak SE-graph}.

So now we are given a colored simple cycle $D$ in $G$, i.e., $D$ is
representable as a concatenation $\bar D_1\circ D_2\circ\ldots \circ \bar
D_{k-1}\circ D_k$, where each $D_i$ is a directed path in $G$; a path
(\emph{snake}) $D_i$ with $i$ odd (even) is colored white (resp. black). Let
$d_1,\ldots,d_k$ be the sequence of bends in $D$, i.e., $d_i$ is a common
endvertex of $D_{i-1}$ and $D_i$ (letting $D_0:=D_k$). We assume that $D$ is
oriented according to the direction of $D_i$ with $i$ even. When this
orientation is clockwise (counterclockwise) around a point in the open bounded
region $O_D$ of the plane surrounded by $D$, we say that $D$ is
\emph{clockwise} (resp. \emph{counterclockwise}). In particular, the cycle
arising from the above path $Z$ is clockwise.

Our goal is to prove the following
  \begin{lemma} \label{lm:gammaD}
Let $D$ be a colored simple cycle in a weak SE-graph $G$. If $D$ is clockwise
then $\gamma_D=2$. If $D$ is counterclockwise then $\gamma_D=-2$.
  \end{lemma}
  \begin{proof}
~We use induction on the number $\eta(D)$ of bends of $D$. It suffices to
consider the case when $D$ is clockwise (since for a counterclockwise cycle
$D'=\bar D'_1\circ D'_2\circ\ldots \circ \bar D'_{k-1}\circ D'_k$, the reversed
cycle $\bar D'=\bar D'_k\circ D'_{k-1}\circ\ldots \circ \bar D'_2\circ D'_1$ is
clockwise, and it is easy to see that $\gamma_{\bar D'}=-\gamma_{D'}$).

W.l.o.g., one may assume that the coordinates $\beta(d_i)$ of all bends $d_i$
are different (as we can make, if needed, a due small perturbation on $D$,
which does not affect $\gamma$).

If $\eta(D)=2$, then $D=\bar D_1\circ D_2$, and the clockwise orientation of
$D$ implies that the path $D_1$ is lower than $D_2$. So
$\gamma(d_1)=\gamma(d_2)=1$, implying $\gamma_D=2$.

 Now assume that $\eta(D)>2$. Then at least one of the following is true:
 \smallskip

(a) there exists a peak $d_i$ such that the horizontal line through $d_i$ meets
$D$ on the left of $d_i$, i.e., there is a point $x$ in $D$ with
$\alpha(x)<\alpha(d_i)$ and $\beta(x)=\beta(d_i)$;

(b) there exists a pit $d_i$ such that the horizontal line through $d_i$ meets
$D$ on the right of $d_i$.
 \smallskip

(This can be seen as follows. Let $d_j$ be a peak with $\beta(d_j)$ maximum. If
$\beta(d_{j-1})\le \beta(d_{j+1})$, then, by easy topological reasonings,
either the pit $d_{j+1}$ is as required in~(b) (when $d_{j+2}$ is on the right
from $D_{j+1}$), or the peak $d_{j+2}$ is as required in~(a) (when $d_{j+2}$ is
on the left from $D_{j+1}$), or both. And if $\beta(d_{j-1})> \beta(d_{j+1})$,
similar properties hold for $d_{j-1}$ and $d_{j-2}$.)

We may assume that case~(a) takes place (for case~(b) is symmetric to~(a)).
Choose the point $x$ as in~(a) with $\alpha(x)$ maximum and draw the horizontal
line-segment $L$ connecting the points $x$ and $d_i$. Then the interior of $L$
does not meet $D$. Two cases are possible:
  \smallskip

(I) $\Inter(L)$ is contained in the region $O_D$; or
 \smallskip

(O) $\Inter(L)$ is outside $O_D$.
\smallskip

Since $x$ cannot be a bend of $D$ (in view of $\beta(x)=\beta(d_i)$ and
$\beta(d_i)\ne\beta(d_{i'})$ for any $i'\ne i$), $x$ is an interior point of
some snake $D_j$; let $D'_j$ and $D''_j$ be the parts of $D_j$ from $s_{D_j}$
to $x$ and from $x$ to $t_{D_j}$, respectively. Using the facts that $D$ is
oriented clockwise and this orientation is agreeable with the forward
(backward) direction of each black (resp. white) snake, one can conclude that
  \begin{numitem1} \label{eq:casesIO}
(a) in case (I), ~$D_j$ is white and $\gamma(d_i)=-1$ (i.e., for the white
snake $D_i$ and black snake $D_{i+1}$ that share the peak $d_i$, ~$D_{i+1}$ is
lower than $D_i$); and (b) in case~(O), ~$D_j$ is black and $\gamma(d_i)=1$
(i.e., $D_i$ is lower than $D_{i+1}$)
  \end{numitem1}

See the picture (where the orientation of $D$ is indicated):

\vspace{-0.3cm}
\begin{center}
\includegraphics{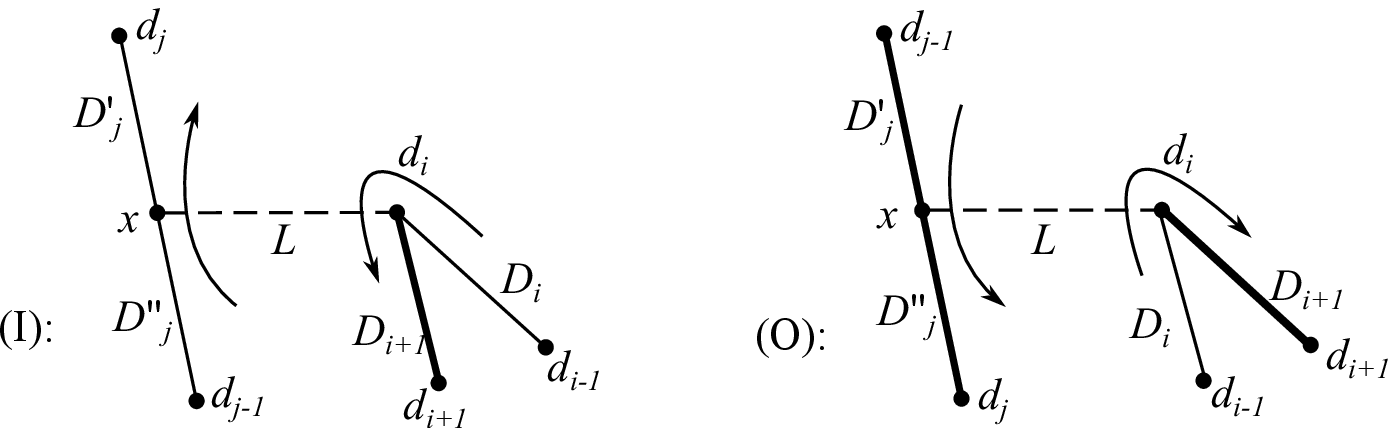}
\end{center}
\vspace{0cm}

The points $x$ and $d_i$ split the cycle (closed curve) $D$ into two parts
$\zeta',\zeta''$, where the former contains $D'_j$ and the latter does $D''_j$.

We first examine case (I). The line $L$ divides the region $O_D$ into two parts
$O'$ and $O''$ lying above and below $L$, respectively. Orienting the curve
$\zeta'$ from $x$ to $d_i$ and adding to it the segment $L$ oriented from $d_i$
to $x$, we obtain closed curve $D'$ surrounding $O'$. Note that $D'$ is
oriented clockwise around $O'$. We combine the paths $D'_j$, $L$ (from $x$ to
$d_i$) and $D_i$ into one directed path $A$ (going from $s_{D'_j}=s_{D_j}=d_j$
to $t_{D_i}=d_{i-1}$). Then $D'$ turns into a correctly colored simple cycle in
which $A$ is regarded as a white snake and the white/black snakes structure on
the rest preserves (cf.~\refeq{casesIO}(a)).

In its turn, the curve $\zeta''$ oriented from $d_{i}$ to $x$ plus the segment
$L$ (oriented from $x$ to $d_i$) form closed curve $D''$ that surrounds $O''$
and is oriented clockwise as well. We combine $L$ and $D_{i+1}$ into one black
snake $B$ (going from $x$ to $d_{i+1}$). Then $D''$ becomes a correctly colored
cycle, and $x$ is a peak in it. (The point $x$ turns into a vertex of $G$.) We
have $\gamma(x)=1$ (since the white $D''_j$ is lower than the black $B$).

The creation of $D',D''$ from $D$ in case (I) is illustrated in the picture:

\vspace{-0.0cm}
\begin{center}
\includegraphics{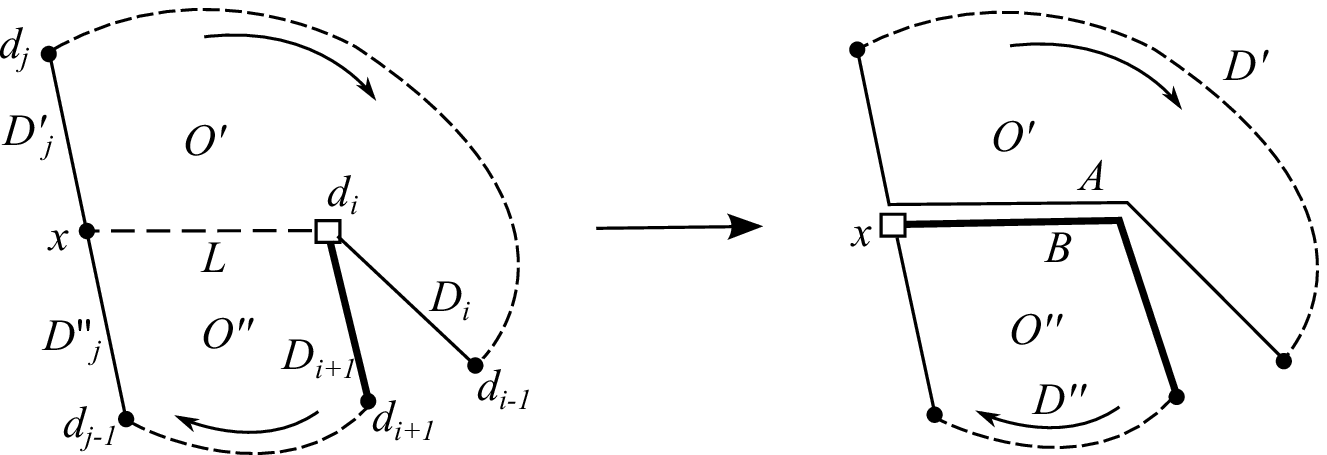}
\end{center}
\vspace{-0.1cm}

We observe that, compared with $D$, the pair $D',D''$ misses the bend $d_i$
(with $\gamma(d_i)=-1$) but acquires the bend $x$ (with $\gamma(d)=1$). Then
   \begin{equation}  \label{eq:DD'D''}
   \eta(D)=\eta(D')+\eta(D''),
   \end{equation}
implying $\eta(D'),\eta(D'')<\eta(D)$. Therefore, we can apply induction. This
gives $\gamma_{D'}=\gamma_{D''}=2$. Now, by reasonings above,
  $$
  \gamma_D=\gamma_{D'}+\gamma_{D''}+\gamma(d_i)-\gamma(x)=2+2-1-1=2,
  $$
as required.

Next we examine case~(O). From the fact that $D$ simple one can conclude that
the curve $\zeta'$ (containing $D'_j$) passes through the black snake
$D_{i+1}$, and the curve $\zeta''$ (containing $D''_j$) through the white snake
$D_i$. Adding to each of $\zeta',\zeta''$ a copy of $L$, we obtain closed
curves $D',D''$, respectively, each inheriting the orientation of $D$. They
become correctly colored simple cycles when we combine the paths
$D'_j,L,D_{i+1}$ into one black snake (from $d_{j-1}$ to $d_{i+1}$) in $D'$,
and combine the paths $L,D_i$ into one white snake (from the new bend $x$ to
$d_i$) in $D''$. Let $O',O''$ be the bounded regions in the plane surrounded by
$D',D''$, respectively. It is not difficult topological exercise to see that
two cases are possible:
  \smallskip

  (O1) ~$O'$ includes $O''$ (and $O_D$);
  \smallskip

  (O2) ~$O''$ includes $O'$ (and $O_D$).

These cases are illustrated in the picture:

\vspace{-0.0cm}
\begin{center}
\includegraphics{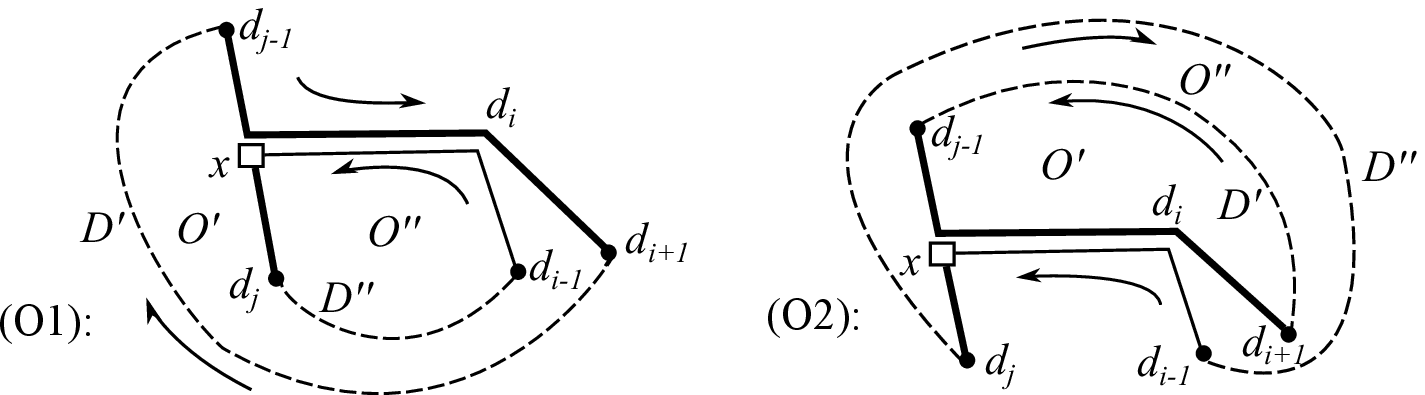}
\end{center}
\vspace{0cm}

Then in case~(O1), ~$D'$ is clockwise and $D''$ is counterclockwise, whereas in
case~(O2) the behavior is converse. Also $\gamma(d_i)=1$ and $\gamma(x)=-1$.
Similar to case~(I), \refeq{DD'D''} is true and we can apply induction. Then in
case~(O1), we have $\gamma_{D'}=2$ and $\gamma_{D''}=-2$, whence
   $$
  \gamma_D=\gamma_{D'}+\gamma_{D''}+\gamma(d_i)-\gamma(x)=2-2+1-(-1)=2.
  $$
And in case~(O2), we have $\gamma_{D'}=-2$ and $\gamma_{D''}=2$, whence
   $$
  \gamma_D=\gamma_{D'}+\gamma_{D''}+\gamma(d_i)-\gamma(x)=-2+2+1-(-1)=2.
  $$

Thus, in all cases we obtain $\gamma_D=2$, yielding the lemma.
 \end{proof}

This completes the proof of Proposition~\ref{pr:seg-seg}. \hfill\qed

%----------------
 \subsection{Proof of Proposition~\ref{pr:seg-link}.} \label{ssec:prop3}

Consider a link $L$. By Lemma~\ref{lm:varphi=1}, for any snake $P$,
~$\varphi_{L,P}\ne 1$ is possible only if $L$ and $P$ have a common endvertex
$v$. Note that $v\notin R\cup C$. In particular, it suffices to examine only
bounded and semi-bounded links.

First assume that $s_L\notin R$. Then there are exactly two snakes containing
$s_L$, namely, a white snake $A$ and a black snake $B$ such that $s_L=t_A=t_B$.
If $L$ is white, then $A$ and $L$ belong to the same path in $\phi$; therefore,
$A\prec L\prec B$. Under the exchange operation $A$ becomes black, $B$ becomes
white, and $L$ continues to be white. Then $B,L$ belong to the same path in
$\psi$; this implies $B\precast L\precast A$. So both pairs $(A,L)$ and $(L,B)$
are permuting, and Lemma~\ref{lm:1atP=asQ} gives $\varphi_{A,L}=q$ and
$\varphi_{L,B}=\bar q$, whence $\varphi_{A,L}\varphi_{L,B}=1$.

Now let $L$ be black. Then $A\prec B\prec L$ and $B\precast A\precast L$. So
both pairs $\{A,L\}$ and $\{B,L\}$ are invariant, whence
$\varphi_{A,L}=\varphi_{B,L}=1$.

The end $t_L$ is examined in a similar way. Assuming $t_L\notin C$, there are
exactly two snakes, a white snake $A'$ and a black snake $B'$, that contain
$t_L$, namely: $t_L=s_{A'}=s_{B'}$. If $L$ is white, then $L\prec A'\prec B'$
and $L\precast B'\precast A'$. Therefore, $\{L,A\}$ and $\{L,B'\}$ are
invariant, yielding $\varphi_{L,A'}=\varphi_{L,B'}=1$. And if $L$ is black,
then $A'\prec L\prec B'$ and $B'\precast L\precast A'$. So both $(A',L)$ and
$(L,B')$ are permuting, and we obtain from Lemma~\ref{lm:1atP=asQ} that
$\varphi_{A',L}=\bar q$ and $\varphi_{L,B'}=q$, yielding
$\varphi_{A',L}\varphi_{L,B'}=1$.

These reasonings prove the proposition. \hfill\qed

%----------------
 \subsection{Degenerate case.} \label{ssec:degenerate}

We have proved relation~\refeq{Pi=q} in a non-degenerate case, i.e., subject
to~\refeq{nondegenerate}, and now our goal is to prove~\refeq{Pi=q} when the
set
  $$
  \Zscr:=\{z_1,\ldots,z_{k-1}\}\cup \{c_j\colon j\in J\cup J'\}
  $$
contains distinct elements $u,v$ with $\alpha(u)=\alpha(v)$. We say that such
$u,v$ form a \emph{defect pair}. A special defect pair is formed by twins
$z_i,z_j$ (bends satisfying $i\ne j$, ~$\alpha(z_i)=\alpha(z_j)$ and
$\beta(z_i)=\beta(z_j)$). Another special defect pair is of the form
$\{s_P,t_P\}$ when $P$ is a \emph{vertical} snake or link, i.e.,
$\alpha(s_P)=\alpha(t_P)$.

We will show~\refeq{Pi=q} by induction on the number of defect pairs.

Let $a$ be the \emph{minimum} number such that the set $X:=\{u\in \Zscr\;\colon
\alpha(u)=a\}$ contains a defect pair. We denote the elements of $X$ as
$v_0,v_1,\ldots,v_r$, where for each $i$, ~$v_{i-1}$ is \emph{higher} than
$v_i$, which means that either $\beta(v_{i-1})>\beta(v_i)$, or $v_{i-1},v_i$
are twins and $v_{i-1}$ is a pit (while $v_{i}$ is a peak) in the exchange path
$Z$. The highest element $v_0$ in this order is also denoted by $u$.

In order to conduct induction, we deform the graph $G$ within a sufficiently
narrow vertical strip $S=[a-\eps,a+\eps]\times \Rset$ (where $0<\eps<
\min\{|\alpha(z)-a|\colon z\in \Zscr-X\}$) to get rid of the defect pairs
involving $u$ in such a way that the configuration of snakes/links in the
arising graph $\tilde G$ remains ``equivalent'' to the initial one. More
precisely, we shift the bend $u$ at a small distance ($<\eps$) to the left,
keeping the remaining elements of $\Zscr$; then the bend $u'$ arising in place
of $u$ satisfies $\alpha(u')<\alpha(u)$ and $\beta(u')=\beta(u)$. The
snakes/links with an endvertex at $u$ are transformed accordingly; see the
picture for an example.

\vspace{-0cm}
\begin{center}
\includegraphics{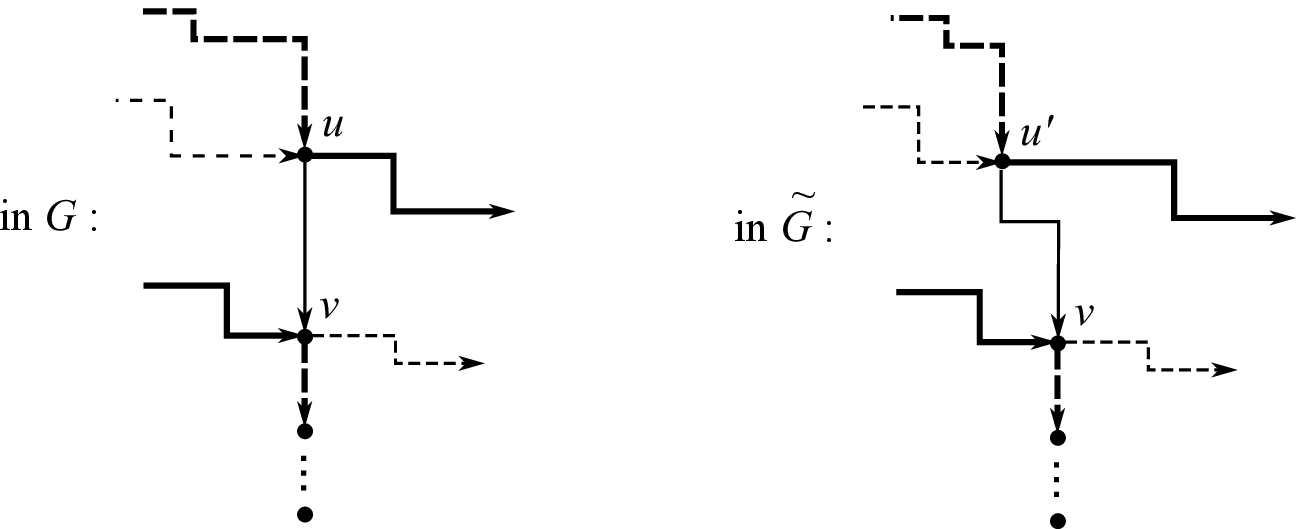}
\end{center}
\vspace{-0.2cm}

Let $\varPi$ and $\tilde\varPi$ denote the L.H.S. value in~\refeq{Pi=q} for the
initial and new configurations, respectively. Under the deformation, the number
of defect pairs becomes smaller, so we may assume by induction that $\varPi=q$.
Thus, we have to prove that
   \begin{equation} \label{eq:varPi}
 \varPi=\tilde\varPi.
  \end{equation}

We need some notation and conventions. For $v\in X$, the set of (initial)
snakes and links with an endvertex at $v$ is denoted by $\Pscr_v$. For
$U\subseteq X$, ~$\Pscr_U$ denotes $\cup(\Pscr_v\;\colon v\in U)$.
Corresponding objects for the deformed graph $\tilde G$ are usually denoted
with tildes as well; e.g.: for a path $P$ in $G$, its image in $\tilde G$ is
denoted by $\tilde P$; the image of $\Pscr_v$ is denoted by $\tilde \Pscr_{v}$
(or $\tilde \Pscr_{\tilde v}$), and so on. The set of standard paths in
$\Pscr_U$ (resp. $\tilde\Pscr_U$) is denoted by $\Pscr^{\rm st}_U$ (resp.
$\tilde\Pscr^{\rm st}_U$). Define
  \begin{equation} \label{eq:u_X-u}
  \varPi_{u,X-u}:=\prod(\varphi_{P,Q}\colon P\in\Pscr_u,\;
Q\in\Pscr_{X-u}).
  \end{equation}
A similar product for $\tilde G$ (i.e., with $\tilde\Pscr_u$ instead of
$\Pscr_u$) is denoted by $\tilde\varPi_{u,X-u}$ .

Note that~\refeq{varPi} is equivalent to
  \begin{equation} \label{eq:varPiX}
  \varPi_{u,X-u}=\tilde\varPi_{u,X-u}.
    \end{equation}
This follows from the fact that for any paths $P,Q\in\Sscr\cup\Lscr$ different
from those involved in~\refeq{u_X-u}, the values $\varphi_{P,Q}$ and
$\varphi_{\tilde P,\tilde Q}$ are equal. (The only nontrivial case arises when
$P,Q\in\Pscr_u$ and $Q$ is vertical (so $\tilde Q$ becomes standard). Then
$t_Q=v_1$. Hence $Q\in \Pscr_{X-u}$, the pair $P,Q$ is involved in
$\varPi_{u,X-u}$, and the pair $\tilde P,\tilde Q$ in $\tilde\varPi_{u,X-u}$.)

To simplify our description technically, one trick will be of use. Suppose that
for each standard path $P\in\Pscr^{\rm st}_X$, we choose a point (not
necessarily a vertex) $v_P\in\Inter(P)$ in such a way that
$\alpha(s_P)<\alpha(v_P)<\alpha(t_P)$, and the coordinates $\alpha(v_P)$ for
all such paths $P$ are different. Then $v_P$ splits $P$ into two subpaths
$P',P''$, where we denote by $P'$ the subpath connecting $s_P$ and $v_P$ when
$\alpha(s_P)=a$, and connecting $v_P$ and $t_P$ when $\alpha(t_P)=a$, while
$P''$ is the rest. This provides the following property: for any
$P,Q\in\Pscr^{\rm st}_X$, ~$\varphi_{P',Q''}=\varphi_{Q',P''}=1$ (in view of
Lemma~\ref{lm:varphi=1}). Hence
$\varphi_{P,Q}=\varphi_{P',Q'}\varphi_{P'',Q''}$. Also $P''=\tilde P''$. It
follows that~\refeq{varPiX} would be equivalent to the equality
  $$
  \prod(\varphi_{P',Q'}\colon P\in\Pscr_u,\;Q\in\Pscr_{X-\{u\}})
  =\prod(\varphi_{\tilde P',\tilde Q'}\colon
  P\in\Pscr_u,\;Q\in\Pscr_{X-\{u\}}).
  $$

In light of these reasonings, it suffices to prove~\refeq{varPiX} in the
special case when
  \begin{numitem1} \label{eq:assumption}
any $P\in\Pscr_u$ and $Q\in\Pscr_{X-u}$ satisfy
$\{\alpha(s_P),\alpha(t_P)\}\cap \{\alpha(s_Q),\alpha(t_Q)\}=\{a\}$.
  \end{numitem1}

For $i=0,\ldots,r$, we denote by $A_i,B_i,K_i,L_i$, respectively, the white
snake, black snake, white link, and black link, that have an endvertex at
$v_i$. Note that if $v_{i-1},v_i$ are twins, then the fact that $v_{i-1}$ is a
pit implies $A_{i-1},B_{i-1}$ are the snakes entering $v_{i-1}$, and $A_i,B_i$
are the snakes leaving $v_i$; for convenience, we formally define $K_{i-1}=K_i$
and $L_{i-1}=L_i$ to be the trivial paths consisting of the the same single
vertex $v_i$. Note that if $v_r\in C$, then some paths among $A_k,B_k,K_k,L_k$
vanish (e.g., both snakes and one link).

When vertices $v_i$ and $v_{i+1}$ are connected by a (vertical) path in
$\Sscr\cup \Lscr$, we denote such a path by $P_i$ and say that the vertex $v_i$
is \emph{open}; otherwise $v_i$ is said to be closed. Note that $v_i,v_{i+1}$
can be connected by either one snake, or one link, or two links (namely,
$K_i,L_i$); in the latter case $P_i$ is chosen arbitrarily among them. In
particular, if $v_i,v_{i+1}$ are twins, then $v_i$ is open and the role of
$P_i$ is played by any of the trivial links $K_i,L_i$. Obviously, in a sequence
of vertical paths $P_i,P_{i+1},\ldots,P_{j}$, the snakes and links alternate.
One can see that if $P_i$ is a white snake, i.e., $P_i=A_i=A_{i+1}=:A$, then
both black snakes $B_i,B_{i+1}$ are standard, and there holds $v_i=s_{B_i}$ and
$v_{i+1}=t_{B_{i+1}}$. See the left fragment of the picture:

\vspace{-0.3cm}
\begin{center}
\includegraphics{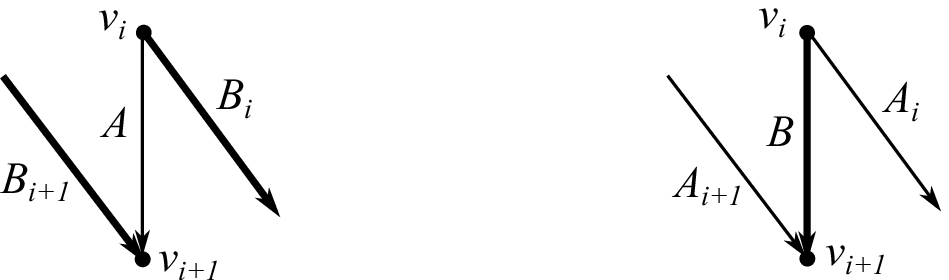}
\end{center}
\vspace{0cm}

Symmetrically, if $P_i$ is a black snake: $B_i=B_{i+1}=:B$, then the white
snakes $A_i,A_{i+1}$ are standard, $v_i=s_{A_i}$ and $v_{i+1}=t_{A_{i+1}}$; see
the right fragment of the above picture.

In its turn, if $P_i$ is a nontrivial white link, i.e., $P_i=K_i=K_{i+1}$, then
two cases are possible: either the black links $L_i,L_{i+1}$ are standard,
$v_i=s_{L_i}$ and $v_{i+1}=t_{L_{i+1}}$, or $L_i=L_{i+1}=P_i$. And if $P_i$ is
a black link, the behavior is symmetric. See the picture:

\vspace{-0.3cm}
\begin{center}
\includegraphics{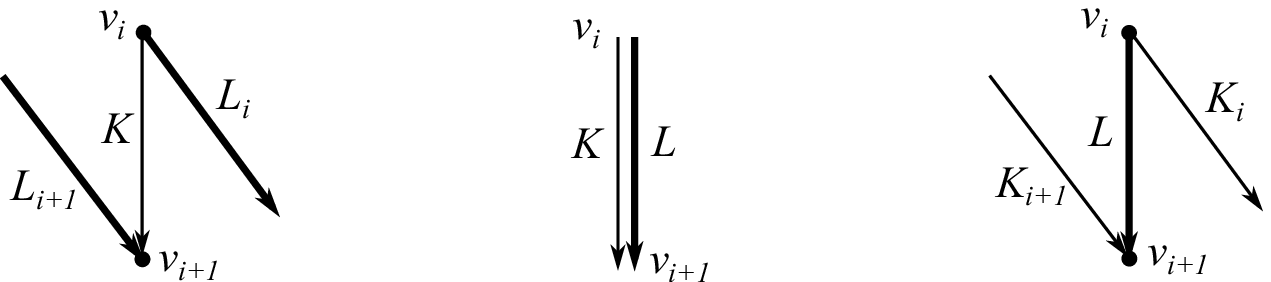}
\end{center}
\vspace{0cm}

Now we are ready to start proving equality~\refeq{varPiX}. Note that the
deformation of $G$ changes none of the orders $\prec$ and $\precast$.

We say that paths $P,P'\in\Pstan_X$ are \emph{separated} (from each other) if
they are not contained in the same path of any of the flows
$\phi,\phi',\psi,\psi'$. The following observation will be of use:
  \begin{numitem1}  \label{eq:monochPQ}
if $P,P'\in\Pstan_X$ have the same color, are separated, and $P'$ is lower than
$P$, then $P'\prec P$; and similarly w.r.t. the order $\precast$ (concerning
$\psi,\psi'$).
  \end{numitem1}
Indeed, suppose that $P,P'$ are white, and let $Q$ and $Q'$ be the paths of the
flow $\phi$ containing $P$ and $P'$, respectively. Since $P,P'$ are separated,
the paths $Q,Q'$ are different. Moreover, the fact that $P'$ is lower than $P$
implies that $Q'$ is lower than $Q$ (taking into account that $Q,Q'$ are
disjoint). Thus, $Q'$ precedes $Q$ in $\phi$, yielding $P'\prec P$, as
required. When $P,P'$ concern one of $\phi',\psi,\psi'$, the argument is
similar.
 \smallskip

In what follows we will use the abbreviated notation $A,B,K,L$ for the paths
$A_0,B_0,K_0,L_0$ (respectively) having an endvertex at $u=v_0$. Also for
$R\in\Pscr_{X-u}$, we denote the product $\varphi_{A,R}\varphi_{B,R}
\varphi_{K,R}\varphi_{L,R}$ by $\varPi(R)$, and denote by $\tilde \varPi(R)$ a
similar product for the paths $\tilde A,\tilde B,\tilde K, \tilde L, \tilde R$
(concerning the deformed graph $\tilde G$). One can see that $\varPi_{u,X-u}$
(resp. $\tilde \varPi_{u,X-u}$) is equal to the product of the values
$\varPi(R)$ (resp. $\tilde\varPi(R)$) over $R\in\Pscr_{X-u}$.

To show~\refeq{varPiX}, we will examine several cases. First of all we consider
 \smallskip

\noindent\underline{\emph{Case (R1)}:} ~$\{u\}$ is closed; in other words, all
paths $A,B,K,L$ are standard (taking into account that $u$ is the highest
vertex in $X$).

  \begin{prop}  \label{pr:caseR1}
In case~(R1), ~$\varPi(R)=\tilde\varPi(R)=1$ holds for any $R\in\Pscr_{X-u}$.
As a consequence, \refeq{varPiX} is valid.
  \end{prop}
  \begin{proof}
~Let $R\in \Pscr_{v_p}$ for $p\ge 1$. Observe that~\refeq{assumption} together
with the fact that the vertex $u$ is shifted under the deformation of $G$
implies that $\{\alpha(s_{\tilde P}),\alpha(t_{\tilde P})\} \cap
\{\alpha(s_{\tilde R}),\alpha(t_{\tilde R})\}=\emptyset$ holds for any $P\in
\Pscr_u$. This gives $\tilde\varPi(R)=1$, by Lemma~\ref{lm:varphi=1}.

Next we show the equality $\varPi(R)=1$. One may assume that $R$ is standard
(otherwise the equality is trivial). It is easy to see that in case~(R1), each
of $A,B,K,L$ is separated from $R$.

Note that $A,B,K,L,R$ are as follows: either (a) $t_A=t_B=s_K=s_L$ or (b)
$s_A=s_B=t_K=t_L$, and either (c) $\alpha(s_R)=a$ or (d) $\alpha(t_R)=a$. Let
us examine the possible cases when the combination of~(a) and~(d) takes place.
  \smallskip

1) Let $R$ be a white link, i.e., $R=K_p$. Since $R$ is white and lower than
$A,B,K,L$, we have $R\prec A,B,K,L$ (cf.~\refeq{monochPQ}). Under the exchange
operation (which, as we know, changes the colors of snakes and preserves the
colors of links), $R$ remains white. Then $R\precast A,B,K,L$. Therefore, all
pairs $\{P,R\}$ with $P\in\Pscr_u$ are invariant, and $\varPi(R)=1$ is trivial.
 \smallskip

2) Let $R=L_p$. Since $R$ is black, we have $A,K\prec R\prec B,L$. The exchange
operation changes the colors of $A,B$ and preserves the ones of $K,L,R$. Hence
$B,K\precast R\precast A,L$, giving the permuting pairs $(A,R)$ and $(R,B)$.
Lemma~\ref{lm:atP=atQ} applied to these pairs implies $\varphi_{A,R}=\bar q$
and $\varphi_{R,B}=q$. Then $\varPi(R)=\varphi_{A,R}\varphi_{R,B}=\bar q q=1$.
\smallskip

3) Let $R=A_p$. Then $R\prec A,B,K,L$ and $B,K\precast R\precast A,L$ (since
the exchange operation changes the colors of $A,B,R$ but not $K,L$). This gives
the permuting pairs $(R,B)$ and $(R,K)$. Then $\varphi_{R,B}=q$, by
Lemma~\ref{lm:atP=atQ}, and $\varphi_{R,K}=\bar q$ by Lemma~\ref{lm:2atP=asQ},
and we have $\varPi(R)=\varphi_{R,B}\varphi_{R,K}=1$.
 \smallskip

4) Let $R=B_p$. (In fact, this case is symmetric to the previous one, as it is
obtained by swapping $(\phi,\phi')$ and $(\psi,\psi')$. Yet we prefer to give a
proof in detail.) We have $A,K\prec R\prec B,L$ and $R\precast A,B,K,L$, giving
the permuting pairs $(A,R)$ and $(K,R)$. Then $\varphi_{A,R}=\bar q$, by
Lemma~\ref{lm:atP=atQ}, and $\varphi_{K,R}=q$, by Lemma~\ref{lm:2atP=asQ},
whence $\varPi(R)=1$.
 \smallskip

The other combinations, namely,~(a) and~(c), ~(b) and~(c), ~(b) and~(d), are
examined in a similar way (where we appeal to appropriate lemmas from
Sect.~\SEC{two_paths}, and we leave this to the reader as an exercise.
\end{proof}

Next we consider
\smallskip

\noindent\underline{\emph{Case (R2)}:} ~$u$ is open; in other words, at least
one path among $A,B,K,L$ is vertical (going from $u$ to $v_1$).
  \smallskip

It falls into several subcases examined in propositions below.

 \begin{prop} \label{pr:caseR2_sep}
In case~(R2), let $R\in\Pstan_{X-u}$ be separated from $A,B,K,L$. Then
$\varPi(R)=\tilde\varPi(R)$.
  \end{prop}
  \begin{proof}
~We first assume that $u=v_0$ and $v_1$ are connected by exactly one path $P_0$
(which may be any of $A,B,K,L$) and give a reduction to the previous
proposition, as follows.

Suppose that we replace $P_0$ by a standard path $P'$ of the same color and
type (snake or link) such that $s_{P'}=u$ (and $\alpha(t_{P'})<a$). Then the
set $\Pscr'_u:=(\{A,B,K,L\}-\{P_0\})\cup\{P'\}$ becomes as in case~(R1), and by
Proposition~\ref{pr:caseR1}, the corresponding product $\Pi'(R)$ of values
$\varphi_{R,P}$ over $P\in\Pscr'_u$ is equal to 1. (This relies on the fact
that $R$ is separated from $A,B,K,L$, which implies validity of~\refeq{varPiX}
for $R$ and corresponding $P\in \Pscr'_u$.)

Now compare the effects from $P'$ and $\tilde P_0$. These paths have the same
color and type, and both are separated from, and higher than $R$. Also
$\alpha(s_{P'})=\alpha(t_{\tilde P_0})=a$ (since $s_{P'}=u$ and $t_{\tilde
P_0}=v_1$). Then using appropriate lemmas from Sect.~\SEC{two_paths}, one can
conclude that $\{\varphi_{R,P'},\varphi_{R,\tilde P_0}\}=\{q,\bar q\}$.
Therefore,
   $$
   \tilde\varPi(R)=\varphi_{R,\tilde P_0}=\varPi'(R)\varphi^{-1}_{R,P'} =\varPi(R).
   $$

Now let $u$ and $v_1$ be connected by two paths, namely, by $K,L$. We again can
appeal to Proposition~\ref{pr:caseR1}. Consider $\Pscr''_u:=\{A,B,K'',L''\}$,
where $K'',L''$ are standard links (white and black, respectively) with
$s_{K''}=s_{L''}=u$. Then $\varPi''(R):= \varPi( \varphi_{R,P}\colon
P\in\Pscr''_u)=1$ and $\{\varphi_{R,K''}, \varphi_{R,\tilde
K}\}=\{\varphi_{R,L''}, \varphi_{R,\tilde L}\}=\{q,\bar q\}$, and we obtain
  $$
  \tilde\varPi(R)=\varphi_{R,\tilde K}\varphi_{R,\tilde L}=
  \varPi''(R)\varphi^{-1}_{R,K''}\varphi^{-1}_{R,L''} =\varphi_{R,A}\varphi_{R,B}
   =\varPi(R),
   $$
as required.
 \end{proof}

 \begin{prop} \label{pr:caseR2_nonsep}
In case~(R2), let $R$ be a standard path in $\Pscr_{v_p}$ with $p\ge 1$. Let
$R$ be not separated from at least one of $A,B,K,L$. Then
$\varPi(R)=\tilde\varPi(R)$.
  \end{prop}
  \begin{proof}
We first assume that $P_0$ is the unique vertical path connecting $u$ and $v_i$
(in particular, $u$ and $v_1$ are not twins). Then $R$ is not separated from
$P_0$.

Suppose that $P_0$ and $R$ are contained in the same path of the flow $\phi$;
equivalently, both $P_0,R$ are white and $P_0\prec R$. Then neither $\psi$ nor
$\psi'$ has a path containing both $P_0,R$ (this is easy to conclude from the
fact that one of $R$ and $P_{p-1}$ is a snake and the other is a link).
Consider four possible cases for $P_0,R$.

(a) Let both $P_0,R$ be links, i.e., $P_0=K$ and $R=K_p$. Then $A,K\prec
K_p\prec B,L$ and $K_p\precast B,K,A,L$ (since $K\precast K_p$ is impossible by
the above observation). This gives the permuting pairs $(A,K_p)$ and $(\tilde
K,K_p)$, yielding $\varphi_{A,K_p}=\varphi_{\tilde K,K_p}$.

(b) Let $P_0=K$ and $R=A_p$. Then $A,K\prec A_p\prec B,L$ and $B,K\precast
A_p\precast A,L$. This gives the permuting pairs $(A,A_p)$ and $(A_p,B)$,
yielding $\varphi_{A,A_p}\varphi_{\tilde A_p,B}=1=\varphi_{\tilde K,A_p}$.

(c) Let $P_0=A$ and $R=K_p$. Then $K,A\prec K_p\prec L,B$ and $K_p\precast
K,B,L,A$. This gives the permuting pairs $(K,K_p)$ and $(\tilde A,K_p)$,
yielding $\varphi_{K,K_p}=\varphi_{\tilde A,K_p}$.

(d) Let $P_0=A$ and $R=A_p$. Then $K,A\prec A_p\prec L,B$ and $K,B\precast
A_p\precast L,A$. This gives the permuting pairs $(\tilde A,A_p)$ and $(\tilde
A_p,B)$, yielding $\varphi_{\tilde A,A_p}=\varphi_{A_p,B}$.

In all cases, we obtain $\varPi(R)=\tilde\varPi(R)$.

When $P_0,R$ are contained in the same path in $\phi'$ (i.e., $P_0,R$ are black
and $P_0\prec R$), we argue in a similar way. The cases with $P_0,R$ contained
in the same path of $\psi$ or $\psi'$ are symmetric.

A similar analysis is applicable (yielding $\varPi(R)=\tilde\varPi(R)$) when
$u$ and $v_1$ are connected by two vertical paths (namely, $K,L$) and exactly
one relation among $K\prec R$, $L\prec R$, $K\precast R$ and $L\precast R$
takes place (equivalently: either $K,R$ or $L,R$ are separated, not both).

Finally, let $u$ and $v_1$ be connected by both $K,L$, and assume that $K,R$
are not separated, and similarly for $L,R$. An important special case is when
$p=1$ and $u,v_1$ are twins.

Note that from the assumption it easily follows that $R$ is a snake. If $R$ is
the white snake $A_p$, then we have $A,K\prec A_p\prec B,L$ and
$B,K,A,L\precast A_p$. This gives the permuting pairs $(A,A_p)$ and $(\tilde
K,A_p)$, yielding $\varphi_{A,A_p}=\varphi_{\tilde K,A_p}$ (since
$\alpha(t_A)=\alpha(t_{\tilde K}$)). The case with $R=B_p$ is symmetric. In
both cases, $\varPi(R)=\tilde\varPi(R)$.
  \end{proof}

 \begin{prop} \label{pr:caseR2_P0}
Let $R=P_0$ be the unique vertical path connecting $u$ and $v_1$. Then
$\varPi(R)=\tilde\varPi(R)=1$.
  \end{prop}
  \begin{proof}
~The equality $\varPi(R)=1$ is trivial. To see $\tilde\varPi(R)=1$, consider
possible cases for $R$. If $R=K$, then $\tilde A\prec \tilde K\prec \tilde
B,\tilde L$ and $\tilde B\precast \tilde K\precast \tilde A,\tilde L$, giving
the permuting pairs $(\tilde A,\tilde K)$ and $(\tilde K,\tilde B)$ (note that
$t_{\tilde A}=t_{\tilde B}=s_{\tilde K}=\tilde u$). If $R=L$, then $\tilde
A,\tilde K,\tilde B\prec\tilde L$ and $\tilde B,\tilde K,\tilde A\precast
\tilde L$; so all pairs involving $\tilde L$ are invariant. If $R=A$, then
$\tilde K\prec\tilde A\prec \tilde L,\tilde B$ and $\tilde K,\tilde B,\tilde
L\precast \tilde A$, giving the permuting pairs $(\tilde A,\tilde L)$ and
$(\tilde A,\tilde B)$ (note that $s_{\tilde A}=s_{\tilde B}=t_{\tilde L}=\tilde
u$). And the case $R=B$ is symmetric to the previous one.

In all cases, using appropriate lemmas from Sect.~\SEC{two_paths} (and relying
on the fact that all paths $\tilde A,\tilde B,\tilde K,\tilde L$ are standard),
one can conclude that $\tilde\varPi(R)=1$.
  \end{proof}

 \begin{prop} \label{pr:caseR2_KL}
Let both $K,L$ be vertical. Then $\varPi(K)\varPi(L)=
\tilde\varPi(K)\tilde\varPi(L)=1$.
  \end{prop}
  \begin{proof}
The equality $\varPi(K)\varPi(L)=1$ is trivial. To see
$\tilde\varPi(K)\tilde\varPi(L)=1$, observe that $\tilde A\prec\tilde K\prec
\tilde B\prec \tilde L$ and $\tilde B\precast \tilde K\precast \tilde A\precast
\tilde L$. This gives the permuting pairs $(\tilde A,\tilde K)$ and $(\tilde
K,\tilde B)$. Using Lemma~\ref{lm:1atP=asQ}, we obtain $\varphi_{\tilde
A,\tilde K}=q$ and $\varphi_{\tilde K\tilde B}=\bar q$, and the result follows.
  \end{proof}

Taken together, Propositions~\ref{pr:caseR2_sep}--\ref{pr:caseR2_KL} embrace
all possibilities in case~(R2). Adding to them Proposition~\ref{pr:caseR1}
concerning case~(R1), we easily obtain the desired relation~\refeq{varPiX} in a
degenerate case.

This completes the proof of Theorem~\ref{tm:single_exch} in case~(C), namely,
relation~\refeq{caseC}.  \hfill\qed\qed

%----------------
 \subsection{Other cases.} \label{ssec:othercases}
Let $(I|J),(I'|J'),\phi,\phi',\psi,\psi'$ and $\pi=\{f,g\}$ be as in the
hypotheses of Theorem~\ref{tm:single_exch}. We have proved this theorem in
case~(C), i.e., when $\pi$ is a $C$-couple with $f<g$ and $f\in J$ (see the
beginning of Sect.~\SEC{exchange}). In other words, the exchange path
$Z=P(\pi)$, used to transform the initial double flow $(\phi,\phi')$ into the
new double flow $(\psi,\psi')$, connects the sinks $c_f$ and $c_g$ that are
covered by the ``white flow'' $\phi$ and the ``black flow'' $\phi'$,
respectively.

The other possible cases in the theorem are as follows:
  \smallskip

(C1) ~$\pi$ is a $C$-couple with $f<g$ and $f\in J'$;
\smallskip

(C2) ~$\pi$ is an $R$-couple with $f<g$ and $f\in I$;
\smallskip

(C3) ~$\pi$ is an $R$-couple with $f<g$ and $f\in I'$;
\smallskip

(C4) ~$\pi$ is an $RC$-couple with $f\in I$ and $g\in J$;
\smallskip

(C5) ~$\pi$ is an $RC$-couple with $f\in I'$ and $g\in J'$.
\smallskip

Case~(C1) is symmetric to~(C). This means that if double flows $(\phi,\phi')$
and $(\psi,\psi')$ are obtained from each other by applying the exchange
operation using $\pi$ (which, in particular, changes the ``colors'' of both $f$
and $g$), and if one double flow is subject to~(C) (i.e., $f$ concerns the
first, ``white'', flow), then the other is subject to~(C1) (i.e., $f$ concerns
the second, ``black'', flow). Rewriting $w(\phi)w(\phi')=qw(\psi)w(\psi')$
(cf.~\refeq{caseC}) as $w(\psi)w(\psi')=q^{-1}w(\phi)w(\phi')$, we just obtain
the required equality in case~(C1) (where $(\psi,\psi')$ and $(\phi,\phi')$
play the roles of the initial and updated double flows, respectively).

For a similar reasons, case~(C3) is symmetric to~(C2), and~(C5) is symmetric
to~(C4). So it suffices to establish the desired equalities merely in
cases~(C2) and~(C4).

To do this, we appeal to reasonings similar to those in
Sects.~\SSEC{prop1}--\SSEC{degenerate}. More precisely, it is not difficult to
see that descriptions in Sects.~\SSEC{prop1} and \SSEC{prop3} (concerning
link-link and snake-link pairs in $\Nscr$) remain applicable and
Propositions~\ref{pr:link-link} and~\ref{pr:seg-link} are directly extended to
cases~(C2) and~(C4). The method of getting rid of degeneracies developed in
Sect.~\SSEC{degenerate} does work, without any troubles, for~(C2) and~(C4) as
well.

As to the method in Sect~\SSEC{prop2} (concerning snake-snake pairs in
case~(C)), it should be modified as follows. We use terminology and notation
from Sects.~\SSEC{seglink} and~\SSEC{prop2} and appeal to
Lemma~\ref{lm:gammaD}.

When dealing with case~(C2), we represent the exchange path $Z=P(\pi)$ as a
concatenation $Z_1\circ \bar Z_2\circ Z_3\circ\cdots \circ\bar Z_k$, where each
$Z_i$ with $i$ odd (even) is a snake contained in the black flow $\phi'$ (resp.
the white flow $\phi$). Then $Z_1$ begins at the source $r_g$ and $Z_k$ begins
at the source $r_f$. An example with $k=6$ is illustrated in the left fragment
of the picture:

\vspace{0cm}
\begin{center}
\includegraphics{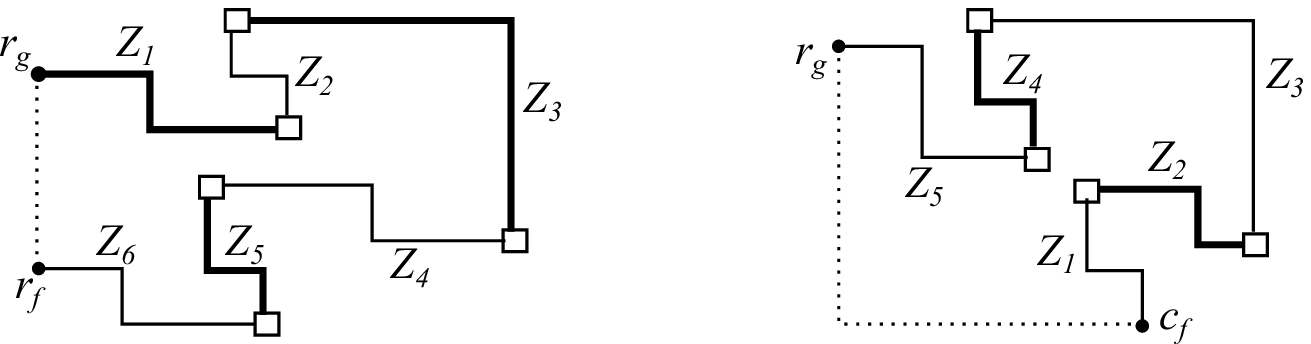}
\end{center}
\vspace{0cm}

The common vertex (bend) of $Z_i$ and $Z_{i+1}$ is denoted by $z_i$. As before,
we associate with a bend $z$ the number $\gamma(z)$ (equal to 1 if, in the pair
of snakes sharing $z$, the white snake is lower that the black one, and $-1$
otherwise), and define $\gamma_Z$ as in~\refeq{gammaZ}. We turn $Z$ into simple
cycle $D$ by combining the directed path $Z_k$ (from $r_f$ to $z_{k-1}$) with
the vertical path from $r_g$ to $r_f$, which is formally added to $G$. (In the
above picture, this path is drawn by a dotted line.) Then, compared with $Z$,
the cycle $D$ has an additional bend, namely, $r_g$. Since the extended white
path $\tilde Z_k$ is lower than the black path $Z_1$, we have $\gamma(r_g)=1$,
and therefore $\gamma_D=\gamma_Z+1$.

One can see that the cycle $D$ is oriented clockwise (where, as before, the
orientation is defined according to that of black snakes). So $\gamma_D=2$, by
Lemma~\ref{lm:gammaD}, implying $\gamma_Z=1$. This is equivalent to the
``snake-snake relation'' $\varphi^{II}=q$, and as a consequence, we obtain the
desired equality
  $$
  w(\phi)w(\phi')=qw(\psi)w(\psi').
  $$

Finally, in case~(C4), we represent the exchange path $Z$ as the corresponding
concatenation $\bar Z_1\circ Z_2\circ \bar Z_3\circ\cdots \circ Z_{k-1}\circ
\bar Z_k$ (with $k$ odd), where the first white snake $Z_1$ ends at the sink
$c_f$ and the last white snake $Z_k$ begins at the source $r_g$. See the right
fragment of the above picture, where $k=5$. We turn $Z$ into simple cycle $D$
by adding a new ``black snake'' $Z_{k+1}$ beginning at $r_g$ and ending at
$c_f$ (it is formed by the vertical path from $r_g$ to $(0,0)$, followed by the
horizontal path from $(0,0)$ to $c_f$; see the above picture). Compared with
$Z$, the cycle $D$ has two additional bends, namely, $r_g$ and $c_f$. Since the
black snake $Z_{k+1}$ is lower than both $Z_1$ and $Z_k$, we have
$\gamma(r_g)=\gamma(c_f)=-1$, whence $\gamma_D=\gamma_Z-2$. Note that the cycle
$D$ is oriented counterclockwise. Therefore, $\gamma_D=-2$, by
Lemma~\ref{lm:gammaD}, implying $\gamma_Z=0$. As a result, we obtain the
desired equality $w(\phi)w(\phi')=w(\psi)w(\psi')$.

This completes the proof of Theorem~\ref{tm:single_exch}.

%---------------------------

\end{document}